\documentclass[a4paper, 11pt]{article}
\usepackage{amsmath, amssymb,amscd, amsthm}
\usepackage[mathscr]{eucal}
\usepackage{graphics}
\usepackage{fullpage}
\newcommand\cyr{%
 \renewcommand\rmdefault{wncyr}%
 \renewcommand\sfdefault{wncyss}%
 \renewcommand\encodingdefault{OT2}%
\normalfont\selectfont} \DeclareTextFontCommand{\textcyr}{\cyr}

\newtheorem{theorem}{Theorem}
\newtheorem{lemma}[theorem]{Lemma}
\newtheorem{corollary}[theorem]{Corollary}
\newtheorem{proposition}[theorem]{Proposition}

\newtheorem{conjecture}[theorem]{Conjecture}

\newtheorem*{theorem3}{Theorem 3}

\def\N{\mathbb N}
\def\Z{\mathbb Z}

\def\T{\mathbb T}

\newcommand{\vectornorm}[1]{\left|\left|#1\right|\right|}

\begin{document}
\title{On the Structure of Sets of Large Doubling}
\author{Allison Lewko \thanks{Supported by a National Defense Science and Engineering Graduate Fellowship} \and Mark Lewko}
\date{}
\maketitle

\begin{abstract} We investigate the structure of finite sets $A \subseteq \Z$ where $|A+A|$ is large. We present a combinatorial construction that serves as a counterexample to natural conjectures in the pursuit of an ``anti-Freiman" theory in additive combinatorics. In particular, we answer a question along these lines posed by O'Bryant. Our construction also answers several questions about the nature of finite unions of $B_2[g]$ and $B^\circ_2[g]$ sets, and enables us to construct a $\Lambda(4)$ set which does not contain large $B_2[g]$ or $B^\circ_2[g]$ sets.
\end{abstract}

\section{Introduction}

Freiman's theorem~\cite{F1} states that if a finite set $A \subseteq \Z$ satisfies $|A+A| \leq \delta |A|$ for some constant $\delta$, then $A$ is contained in a generalized arithmetic progression of dimension $d$ and size $c |A|$, where $c$ and $d$ depend only on $\delta$ and not on $|A|$. One might then ask about the opposite extreme: if $|A+A| \geq \delta |A|^2$, what can one say about the structure of $A$ as a function only of $\delta$?
The natural candidate for the building blocks of such a theory are $B_2[g]$ sets (a set $S \subseteq \Z$ is a $B_2[g]$ set if any integer can be expressed in at most $g$ ways as a sum of two elements in $S$). It is clear that finite $B_2[g]$ sets are sets of large doubling, but to what extent can we describe all sets of large doubling in terms of $B_2[g]$ sets?

A first attempt at an anti-Freiman theory might be to guess that if $|A+A|\geq \delta |A|^2$ for some positive constant $\delta$, then $A$ can be decomposed into a union of $k$ $B_2[g]$ sets where $k$ and $g$ depend only on $\delta$. This is easily shown to be false. For example, one can start with a $B_2[1]$ set of $n$ elements, and take its union with an arithmetic progression with $n$ elements. One then obtains an $A$ such that $|A+A| \geq \delta |A|^2$ for some $\delta$ (independent of $n$), but the arithmetic progression contained in $A$ will not be decomposable into a union of $k$ $B_2[g]$ sets with $k$ and $g$ depending only on $\delta$ as $n$ tends infinity.

There are two ways we might try to fix this problem: first, we might ask only that $A$ \emph{contains} a $B_2[g]$ set of size $\delta' |A|$, where $\delta'$ and $g$ depend only on $\delta$ (this question was posed by O'Bryant in \cite{OB1}). Second, we might ask that $|A'+A'|\geq \delta |A'|^2$ hold for all subsets $A' \subseteq A$ for the same value of $\delta$. Either of these changes would rule out the trivial counterexample given above. However, even applying both of these modifications simultaneously is not enough to make the statement true. We provide a sequence of sets $W_{n,k} \subseteq \Z$ where $|W'+W'|\geq \delta |W'|^2$ holds for all of their subsets $W'$ for the same value of $\delta$, but if we try to express each $W_{n,k}$ as a union of $B_2[g]$ sets for a fixed $g$, we are forced to let the union size tend to infinity as $k$ tends to infinity. Our sequence of sets also fails to contain large $B_2[g]$ sets. (The parameter $n$ will be chosen sufficiently large with respect to $k$ and $g$ for each $k$. We include $n$ here for consistency with our later notation.)

Our initial sets $W_{n,k}$ are $B^\circ_2[2]$ sets (a set $S \subseteq \Z$ is a $B^\circ_2[g]$ set if any nonzero integer can be expressed in at most $g$ ways as a difference of two elements in $S$). This may lead one to make the following weaker anti-Freiman conjecture:
\begin{conjecture}
\label{con:anti-Freiman}
(Weak Anti-Freiman) Suppose that $A \subseteq \Z$ is a finite set that satisfies $|A'+A'|\geq \delta |A'|^2$ and $|A'-A'|\geq \delta |A'|^2$ for all subsets $A' \subseteq A$. Then $A$ contains either a $B_2[g]$ set or a $B^\circ_2[g]$ set of size $\geq \delta' |A|$, where $g$ and $\delta'$ depend only on $\delta$.
\end{conjecture}
We show that even this very weak conjecture is false.

Our approach to obtaining a counterexample starts with constructing a union of $k$ $B_2[g]$ sets that cannot be decomposed as a union of $k-1$ $B_2[g']$ sets for any $g'$. This is related to a problem previously studied, with the roles of $k$ and $g$ reversed: Erd\H{o}s and Newman \cite{EO1} independently conjectured that for every $g\geq 2$, there exists a $B_{2}[g]$ set that is not a finite union of $B_{2}[g-1]$ sets. Erd\H{o}s \cite{EO1} established the conjecture for certain values of $g$ using Ramsey theory, and Ne\v{s}etril and R\"{o}dl~\cite{NE1} proved the conjecture for all values of $g$ using arguments based on Ramsey graphs. Instead of considering $B_2[g]$ sets that are not finite unions of $B_2[g-1]$ sets, we fix $g=1$ and for each $k$, we construct a union of $k$ $B_2[1]$ sets that is not a union of $k-1$ $B_2[g']$ sets for any $g'$. The key feature of our construction is that we can precisely control the form of the repeated sums (elements $a,b,c,d$ in our set such that $a + b = c+d$) and repeated differences ($a-b = c-d$), which allows us to keep the sumsets large as we let the union size $k$ tend to infinity.

Our construction is an explicit combinatorial object with many interesting properties, answering several questions about the nature of finite unions of $B_2[g]$ and $B^\circ_2[g]$ sets. In particular, for each positive integer $k \geq 5$, we construct:
\begin{enumerate}
  \item a $B^\circ_2[2]$ set in $\Z$ which is a union of $k$ $B_2[1]$ sets and cannot be decomposed as a union of $k-1$ $B_2[g]$ sets for any $g$
  \item a $B_2[2]$ set in $\Z$ which is a union of $k$ $B^\circ_2[1]$ sets and cannot be decomposed as a union of $k-1$ $B^\circ_2[g]$ sets for any $g$
  \item a set in $\Z^2$ which is a direct product of a $B_2[2]$ set in $\Z$ and a $B^\circ_2[2]$ set in $\Z$ and which cannot be expressed as a mixed union of $\frac{k}{3}-1$ $B^\circ_2[g]$ and $B_2[g]$ sets in $\Z^2$
\end{enumerate}
(we say mixed union to simply mean that the union can include \emph{both} $B_2[g]$ and $B^\circ_2[g]$ sets).

In \cite{E2}, Erd\H{o}s and S\'{o}s asked if there is a $B^\circ_2[g]$ set which is not a finite union of $B_2[1]$ sets. By a standard argument, our finite $B^\circ_2[2]$ sets for each $k$ can be combined to yield an infinite $B^\circ_2[2]$ set which is not a finite union of $B_2[g]$ sets for any $g$, which provides an answer to this question. In contrast, note that any $B^\circ_2[1]$ set is also a $B_2[1]$ set.

\subsection{Connection to $\Lambda(4)$ sets}
There is a connection between sets of large doubling and $\Lambda(4)$ sets, as illustrated in Lemma \ref{LambdatoSum}. If $S$ is a $\Lambda(4)$ set, then $|A+A|\geq \delta |A|$ holds for all finite subsets $A$ of $S$ where $\delta$ depends only on $S$, and not on the choice of $A$. In his 1960 paper \cite{RU1}, Rudin asked if every $\Lambda(2h)$ set is a finite union of $B_h[g]$ sets (for definitions of $\Lambda(2h)$ sets and $B_h[g]$ sets, see subsection~\ref{sec:definitions}). Rudin's question is natural because any finite union of $B_h[g]$ sets is a $\Lambda(2h)$ set, and most known examples of $\Lambda(2h)$ sets are constructed as finite unions of $B_h[g]$ sets.

Meyer \cite{M1} demonstrated a negative answer to Rudin's question by constructing a set $E \subseteq \Z$ which is a $\Lambda(p)$ set for all $p>2$ and is not a finite union of $B_2[g]$ sets. He let $t_0, t_1, t_2, \ldots $ denote a sequence such that $t_{n+1} \geq 3 t_n$ for all $n$ and let $E := \{ t_n - t_m | 0 \leq m < n\}$. To see this is not a finite union of $B_2[g]$ sets for any $g$, Meyer considers sums of the form:
\[ (t_i - t_j) + (t_j - t_\ell) = t_i - t_\ell,\]
where $\ell < j< i$. Meyer's argument proceeds via a recurrence argument. Alternatively, one can use Ramsey's theorem. We suppose that $E$ is the union of $B_2[g]$ sets $G_1, \ldots, G_k$ for some finite values $g, k$, and we derive a contradiction. We color the pairs of natural numbers with $k$ colors by giving $(i, j)$ the color $c$ when $t_i - t_j \in G_c$ (for $i > j$). A general version of Ramsey's Theorem (which can be found in \cite{D1}, for example) says that there must be an infinite monochromatic set $M \subseteq \mathbb{N}$ (meaning that all pairs $(i,j)$ for $i,j \in M$ have the same color). If we take $\ell, i \in M$ such that there are more than $g$ values $j$ such that $\ell < j < i$ and $j \in M$, then we have more than $g$ ways of representing $t_i - t_\ell$ as sum of two elements from the set $G_c$, where $c$ is the color of $M$. This contradicts that $G_c$ is a $B_2[g]$ set.

Meyer's set $E$ is not a finite union of $B_2[g]$ sets for any $g$, yet for some fixed $\delta$, $|A + A| \geq \delta |A|^2$ for all finite $A \subset E$. However, this does not contradict our weak anti-Freiman conjecture, since finite subsets $A \subseteq E$ may still \emph{contain} large $B_2[g]$ sets. More concretely, if we take $t_n = 5^n$ for all $n$, and $A$ is any finite subset of $E = \{t_n - t_m | 0 \leq m < n\}$, then $A$ must contain a $B_2[2]$ set of size at least $\frac{1}{4} |A|$. To see this, we partition the values $\{t_i\}$ into two disjoint sets: $U$ and $L$. We consider the subset $A'$ of $A$ consisting of values $t_i - t_j$ where $t_i \in U$ and $t_j \in L$. A sum of any two such values, e.g. $t_i - t_j + t_{i'} - t_{j'}$ for $t_i, t_i' \in U$, $t_j, t_j' \in L$, will involve no cancelation because $\{i,i'\} \cap \{j, j'\} = \emptyset$. Since base 5 expansions of integers with coefficients in $\{-2, -1, 0, 1, 2\}$ are unique, we will be able to determine the sets $\{i,i'\}$ and $\{j, j'\}$ from the value of the sum. This leaves only two possible ways of expressing the value as a sum of two elements in $A'$: $(t_i - t_j) + (t_{i'}- t_{j'})$ or $(t_i - t_{j'}) + (t_{i'} - t_j)$. Now, if we independently place each $t_i$ in either $U$ or $L$ randomly (probability 1/2 for each), each element $t_i - t_j$ of $A$ will have probability $\frac{1}{4}$ of ending up in $A'$. By linearity of expectation, this means the expected size of $A'$ is $\frac{1}{4} |A|$. Hence, there must be a choice of $U$ and $L$ for which $|A'| \geq \frac{1}{4} |A|$.

In \cite{AE}, Alon and Erd\H{o}s asked if there exists a set $E$ such that for some fixed $\delta >0$, every finite subset $A \subset E$ contains a $B_2[1]$ set of size at least $\delta |A|$, but $E$ is not a finite union of $B_2[1]$ sets. In \cite{ENR}, Erd\H{o}s, Ne\v{s}etril, and R\"{o}dl constructed such a set using sophisticated techniques. Meyer's set is a simpler construction which has a similar property: we have shown that its subsets contain large $B_2[2]$ sets instead of $B_2[1]$ sets.

Our techniques also give a $\Lambda(4)$ set which is not a finite union of $B_2[g]$ sets, and in fact we obtain a stronger negative result for $\Lambda(4)$ sets. We note that it is natural to consider not only $B_2[g]$ sets, but also $B^\circ_2[g]$ sets, since these are $\Lambda(4)$ sets as well. In light of Meyer's result, one may ask the weaker question: \emph{Does a $\Lambda(4)$ set at least contain a large $B_{2}[g]$ or $B_{2}^{\circ}[g]$ set?} A precise version of this question is stated below (see Theorem \ref{thm:lambda4}). This statement is suggested by the following connection with Sidon sets.

Notice that there is no interesting notion of a $\Lambda(\infty)$ set, since a subset of $\Z$ will be a $\Lambda(\infty)$ set (with the obvious extension of our definition below) if and only if it is finite. However, an often useful substitute for $\Lambda(\infty)$ sets are Sidon sets (Sidon sets are a name also attached to $B_{2}[1]$ sets, but we do not use that convention here). These are sets $S \subset \Z$ satisfying

\[\sum_{\xi \in S}|\hat{f}(\xi)| \leq K_{\infty}(S) \vectornorm{\sum_{\xi \in S}\hat{f}(\xi)e(\xi  x)}_{L^{\infty}},\]
where $K_{\infty} (S)$ is a constant depending on the set $S$.

Clarifying our assertion that Sidon sets play the role of $\Lambda(\infty)$ sets, Pisier \cite{PI1} has shown that $S$ is a Sidon subset of $\Z$ if and only if $\sup_{p>2}\frac{K_{p}(S)}{\sqrt{p}} < \infty$. This can be used to show that finite unions of Sidon sets are Sidon sets. We call a set $S$ independent if, for any distinct set of elements, say $\{s_{1},s_{2},\ldots,s_{h}\}$, there is no choice of $+$'s and $-$'s for each $s_i$ such that

\[\pm s_{1} \pm s_{2} \pm \ldots \pm s_{h} =0. \]

One can show that an independent set is a Sidon set, and hence finite unions of independent sets are Sidon sets. One will notice that the definition of independent is somewhat like a limiting case of the condition that the number of representations of an integer as a sum of $h$ elements of the set (and certain generalizations of this) be bounded as $h$ tends to infinity. In the Sidon setting, an obvious analog of Rudin's question is: \emph{Is every Sidon set a finite union of independent sets?} This question is open (although some progress has been made in other groups), however Pisier has shown that a Sidon set must contain a large independent set in the following sense:

\begin{theorem}If $S \subset \Z$ is a Sidon set, then there exists a constant $\delta > 0$ so that for every finite subset $A \subset S$, there is an independent set $I \subseteq A$ satisfying $|I| \geq \delta |A|$.
\end{theorem}

In light of Pisier's theorem, one might ask if it is the case that a $\Lambda(4)$ set must contain a large $B_{2}[g]$ or $B_{2}^{\circ}[g]$ set. We show that the analog of Pisier's theorem fails in the $\Lambda(4)$ setting:

\begin{theorem}\label{thm:lambda4} There exists a $\Lambda(4)$ set $S \subset \Z$ such that for any fixed choice of $\delta>0$ and $g$, there exists a finite subset $A$ of $S$ such that no subset $A'$ of $A$ satisfying $|A'| \geq \delta |A|$ is a $B_{2}[g]$ or $B_{2}^{\circ}[g]$ set.
\end{theorem}

We note that this result cannot be obtained from Meyer's set $E$, since any finite subset of $E$ contains a large $B_2[2]$ set, as discussed above.

\subsection{Related Work}
We are aware of two other constructions of $\Lambda(4)$ sets which are not known to be finite unions of $B_2[g]$ sets. In \cite{B01}, Bourgain probabilistically proved the existence of a $\Lambda(4)$ set $S$ such that $|[0,n]\cap S| \gg n^{1/2}$ for every $n \in \N$.  A theorem of Erd\H{o}s (see \cite{HA1}, Theorem 8 on page 89) states that if $A$ is a $B_{2}[1]$ set, then

\begin{equation}\label{ErdosDensity}
 |A\cap[0,n]| \ll \frac{n^{1/2}}{\ln^{1/2}(n)}
\end{equation}

for infinitely many $n$. It follows from this that Bourgain's set is not the finite union of $B_{2}[1]$ sets. This observation essentially appears in \cite{B02}. If one could show (for infinitely many $n$) that
\[|A\cap[0,n]| = o(n^{1/2}) \]
whenever $A$ is a $B_{2}[g]$ set, it would follow that Bourgain's set is not a finite union of $B_{2}[g]$ sets. Such strong estimates are not currently known.

In \cite{KL1}, Klemes constructed an example of a $\Lambda(4)$ set using an intricate selection algorithm based on a tree structure. While he was able to establish that his set was a $\Lambda(4)$ set without deciding if his set was a finite union of $B_{2}[g]$ sets, he conjectured that the set could in fact be decomposed in this way.

\subsection{Preliminaries}
\label{sec:definitions}

We now give formal definitions of $B_h[g]$ sets, $B_2^\circ[g]$, and $\Lambda(p)$ sets. We define these for all $2<p<\infty$ and all positive integer values of $h$, although in this paper we will only be concerned with $h=2$ and $p=4$. Below, $d$ denotes a positive integer, and $\Z^{d}$ denotes the additive group of tuples of $d$ integers.

\paragraph{$\bf{B_h[g]}$ sets} A set $S \subseteq \Z^d$ is called a $B_h[g]$ set if the number of representations of every $\xi \in \Z^{d}$ as a sum $\xi=\nu_{1}+\ldots +\nu_{h}$ for $\nu_{1},\ldots,\nu_{h} \in S$ is at most $h! g$. The $h!$ is a matter of notational convenience (essentially, we do not wish to count reorderings of summands separately). In particular, a $B_2[g]$ set in $\mathbb{Z}$ is a set such that any integer can be expressed as a sum of two elements in the set in at most $g$ ways (where exchanging the order of the summands does not count as a new representation). We note that for a $B_2[1]$ set, all sums are unique.

\paragraph{$\bf{B^\circ_2[g]}$ sets} A set $S \subseteq \Z^d$ is called a $B^\circ_2[g]$ set if every nonzero element of $\Z^d$ can be expressed as a difference of two elements of $S$ in at most $g$ ways. (We note that there are always many representations of 0 as $a-a$, $b-b$, and so on.)

\paragraph{$\bf{\Lambda(p)}$ sets}
Let $\T^d$ denote the $d$-dimensional torus.  For a measurable complex-valued function $f$ on $\T^d$, we define its $L^{p}$ norm as $||f||_{L^p}=\left(\int_{\T^d}|f(x)|^{p}dx\right)^{1/p}$. We denote the space of all measurable complex-valued functions on $\T^{d}$ with finite $L^p$ norm as $L^{p}(\T^d)$.  Defining $e(x) := e^{2\pi i x}$, we have that a function $f \in L^{2}(\T^{d})$ can be expressed as a Fourier series

\[ f(x) \approx  \sum_{\xi \in \Z^{d}} \hat{f}(\xi) e(\xi \cdot x). \]

To avoid issues regarding the convergence of the sum defining the series, one could always take $f$ such that $\hat{f}(\xi)$ has finite support (i.e. trigonometric polynomials) in what follows. This restriction suffices since we are interested in establishing $L^{p}$ inequalities, and functions with finitely supported Fourier expansions form a dense subspace of $L^{p}(\T^{d})$. In \cite{RU1}, Rudin defined a subset of $S \subseteq \Z^{d}$ to be a $\Lambda(p)$ set, for $p>2$, if there exists a constant $K_{p}(S)$ such that
\begin{equation}\label{LambdaDef}
\vectornorm{f}_{L^p} \leq K_{p}(S) ||f||_{L^{2}}
\end{equation}
whenever $\text{supp}(\hat{f})\subseteq S$. When we wish to emphasize the dimension $d$ of the set $S$, we will write $K_{p}^{d}(S)$.

When $p$ is an even integer, say $p=2h$, one can expand the left-hand side of (\ref{LambdaDef}) and obtain

\[\vectornorm{f}_{L^{2h}}^{h} = \vectornorm{|f|^h }_{L^{2}} = \left(\sum_{\xi \in \Z^{d}} \left| \sum_{\substack{\xi = \nu_{1}+\ldots + \nu_{h} \\ \nu_{1},\ldots,\nu_{h} \in S}} \hat{f}(\nu_{1})\hat{f}(\nu_{2})\ldots\hat{f}(\nu_{h})  \right|^2    \right)^{1/2}\]

\[\leq \left(\sum_{\xi \in \Z^{d}}\left(R_{h}(\xi,S)\right)^2 \sup_{\substack{\left|\hat{f}(\nu_{1})\hat{f}(\nu_{2})\ldots \hat{f}(\nu_{h})  \right|^2 \\ \xi = \nu_{1}+\ldots + \nu_{h}}}  \left| \hat{f}(\nu_{1})\hat{f}(\nu_{2})\ldots \hat{f}(\nu_{h}) \right|^2 \right)^{1/2} \]

\begin{equation}\label{LambdaRep}
\leq \sup_{\xi \in \Z^d} R_{h}(\xi,S) \left( \sum_{\nu \in \Z^d}|\hat{f}(\nu)|^2 \right)^{h/2} \leq \sup_{\xi \in \Z^d} R_{h}(\xi,S) \vectornorm{f}_{L^2}^{h},
\end{equation}
where $R_{h}(\xi,S)$ denotes the number of representations of $\xi \in \Z^{d}$ as a sum $\xi=\nu_{1}+\ldots +\nu_{h}$ for $\nu_{1},\ldots,\nu_{h} \in S$. Thus any set $S$ with the property that $R_{h}(\xi,S)\leq h! g <\infty$ is a $\Lambda(2h)$ set. In particular, every finite set is a $\Lambda(p)$ set for every $p>2$.

We have now shown that every $B_h[g]$ set is a $\Lambda(2h)$ set. One might ask if every $\Lambda(2h)$ set is a $B_{h}[g]$ set.  This is easily seen to be false.  Notice that the union of two $\Lambda(p)$ sets, say $S= S_{1}\cup S_{2}$, is also a $\Lambda(p)$ set.  Letting $K_{p}(S_{1})$ and $K_{p}(S_{2})$ denote the $\Lambda(p)$ constants of the sets $S_{1}$ and $S_{2}$ respectively, for any $f$ with $\hat{f}$ supported on $S$, the triangle inequality gives:

\[ \vectornorm{f}_{L^p} = \vectornorm{\sum_{\nu_{1}\in S_{1}}\hat{f}(\nu_{1})e(\nu_{1}\cdot x)  +  \sum_{\nu_{2}\in S_{2}\setminus S_1}\hat{f}(\nu_{2})e(\nu_{2}\cdot x) }_{L^p}\]

\begin{equation}\label{LambdaUnion}
\leq \vectornorm{\sum_{\nu_{1}\in S_{1}}\hat{f}(\nu_{1})e(\nu_{1}\cdot x)}_{L^p}  +  \vectornorm{\sum_{\nu_{2}\in S_{2}\setminus S_1}\hat{f}(\nu_{2})e(\nu_{2}\cdot x) }_{L^p} \leq (K_{p}(S_{1})+K_{p}(S_{2})) \vectornorm{f}_{L^2}.
\end{equation}

Now we note that $S_{1}= \{2^{i} : i\in\N\}$ and $S_{2}= \{-2^{j} : j \in \N \}$ are each $B_{2}[1]$ sets but $S_{1}\cup S_{2}$ is not a $B_{2}[g]$ for any finite $g$. The next natural question is Rudin's question: is every $\Lambda(2h)$ set a finite union of $B_{h}[g]$ sets? (Rudin asked this only for dimension $d=1$, but it follows from the methods described below and a standard compactness argument that a counterexample in any dimension can be transformed into a counterexample in every other dimension.) Meyer's counterexample \cite{M1} shows that the answer to this question is no for all $h \geq 2$.

\section{A First Attempt at a Combinatorial Construction}
In~\cite{EO1}, Erd\H{o}s constructed a $B_2[3]$ set that is not a finite union of $B_2[g]$ sets for $g < 3$, which he proved by applying Ramsey theory. He conjectured that for any $g$, there exists a $B_2[g]$ set $A$ that is not a finite union of $B_2[g-1]$ sets. This was later proven for all $g$ by Ne\v{s}etril and R\"{o}dl~\cite{NE1}. Informally, this result means that one cannot always tradeoff a larger union size to obtain a lower value of $g$ when representing a set as a finite union of $B_2[g]$ sets.

Our approach to the anti-Freiman problem is to begin by solving a variant of Erd\H{o}s' problem where the roles of $g$ and the union size are switched. Informally put, we seek to prove that one cannot always tradeoff a higher value of $g$ to obtain a smaller union size when representing a set as a finite union of $B_2[g]$ sets.

As a first attempt, we consider a Ramsey-theoretic approach, much like Erd\H{o}s and somewhat reminiscent of Meyer's set $E$. For each positive integer $k$, we will construct an infinite $S \subseteq \Z$ such that $S$ is a union of $2^k$ $B_2[2^{k-1}]$ sets, but not a union of $2^k-1$ $B_2[g']$ sets for any constant $g'$. The undesirable feature of this construction is that the value of $g$ is a function of $k$. This dependence of $g$ on $k$ is removed from our main construction in the next section, where we are able to fix $g = 1$, but it is instructive to consider this simpler construction first.

\begin{proposition} For every positive integer $k$, there exists a set $S \subseteq \Z$ such that $S$ is a union of $2^k$ $B_2[2^{k-1}]$ sets, and $S$ cannot be decomposed as a union of $2^k -1$ $B_2[g']$ sets for any finite $g'$.
\end{proposition}

\begin{proof} We first define $k$ disjoint sequences of positive integers, $X_1 = \{x^1_i\}_{i=1}^\infty, X_2 = \{x^2_i\}_{i=1}^\infty,$ $\ldots$, $X_k = \{x^k_i\}_{i=1}^\infty$, where each consists of powers of 5. For concreteness, we can take $X_j$ to be the sequence $\{5^{ik+j}\}_{i=1}^{\infty}$ for each $j$.
We note that base 5 expansions of integers with coefficients in $\{-2,-1,0,1,2\}$ are unique.

We let $v_1, \ldots, v_{2^k} \in \{1, -1\}^k$ denote all of the distinct vectors of length $k$ with entries in $\{1,-1\}$. For $j$ from 1 to $k$, we define the set
\[S_j := \{ (x^1, x^2, \ldots, x^k) \cdot v_j \; \big| \; x^1 \in X_1, \ldots, x^k \in X_k\}. \]
We set $S := \bigcup_{j=1}^k S_j$. We note that each element of $S$ has a unique representation as $(x^1, \ldots, x^k) \cdot v_j$ for $x^1 \in X_1, \ldots, x^k \in X_k$ and $1 \leq j \leq 2^{k}$.

We claim that each $S_j$ is a $B_2[g]$ set, for $g = 2^{k-1}$. To see why, we consider adding two elements of $S_j$:
\[(x^1, x^2, \ldots, x^k) \cdot v_j + (y^1, y^2, \ldots, y^k)\cdot v_j = (x^1+y^1, x^2+y^2, \ldots, x^k + y^k)\cdot v_j.\]
Here, $x^1, y^1 \in X_1$, $x^2, y^2 \in X_2, \ldots, x^k, y^k \in X_k$. Recalling that the sequences $X_1, \ldots, X_k$ are disjoint sequences of powers of 5, we see that this is a base 5 expansion of an integer with coefficients in $[-2,2]$ (coefficients of $2$ or $-2$ will appear only where $x^i = y^i$). Since these expansions are unique, this sum uniquely determines the values of $x^1, y^1, x^2, y^2, \ldots, x^k, y^k$, up to exchanges of $x^i$ and $y^i$. In other words, it determines the unordered sets $\{x^i, y^i\}$ for $i$ from 1 to $k$. There are $2^k$ ways to choose two elements of $S_j$ which match these sets: for each set $\{x^i, y^i\}$, we must decide whether $x^i$ will be included in the first or second element. Thus, each $S_j$ is a $B_2[2^{k-1}]$ set.

Now we prove that $S$ cannot be decomposed into $2^k-1$ $B_2[g']$ sets for any $g'$. We suppose that $S$ can be decomposed into $2^k-1$ $B_2[g']$ sets, $A_1, \ldots, A_{2^k-1}$, and proceed to derive a contradiction. We will use this decomposition to give a ${2^k \choose 2}$-coloring of all $k$-element subsets of $\mathbb{N}$.

To color the set $(i_1, \ldots, i_k)$ for $i_1 < i_2 < \ldots < i_k$, we consider the following $2^k$ elements of $S$:
\[(x^1_{i_1},x^2_{i_2}, \ldots, x^k_{i_k})\cdot v_1 \in S_1,\]
\[\vdots\]
\[(x^1_{i_1},x^2_{i_2}, \ldots, x^k_{i_k}) \cdot v_{2^k} \in S_{2^k}.\]
Since we have decomposed $S$ into $2^k-1$ sets, some pair of these elements must belong in the same $A_n$. We color $(i_1, \ldots, i_k)$ according to which pair this is (if several pairs are in the same $A_n$, we choose one arbitrarily). For example, if the element of $S_1$ and the element of $S_2$ are placed in the same $A_n$, we may assign the color corresponding to the pair (1,2).

Since we are coloring $k$-element subsets of $\mathbb{N}$ with finitely many colors, a general version of Ramsey's Theorem (again, this can be found in e.g. \cite{D1}) tells us that there exists an infinite monochromatic set $M \subseteq \mathbb{N}$. This means that for any two $k$-element subsets of $M$, the color assigned to them is the same. We call this single color $c(M)$.

Now, $c(M)$ corresponds to a pair $(i,j)$ of indices between 1 and $2^k$. We note that the corresponding vectors $v_i$ and $v_j$ differ in some coordinate $\ell$ (i.e. $v_i + v_j = 0$ in the $\ell^{th}$ coordinate). We consider $k$-element subsets of $M$: ($m_1 < m_2 < \ldots < m_k)$.

We consider fixing elements of $M$ in the indices $\neq \ell$ and letting the element $m_\ell$ vary over $M$ (while satisfying the ordering condition). For each value of $m_\ell$, we get two corresponding elements of some $A_n$ whose sum is equal to
\[(x^1_{m_1}, \ldots, x^k_{m_k}) \cdot (v_i + v_j),\]
which does not depend on $m_\ell$. Since $M$ is infinite, the number of values of $m_\ell$ satisfying the ordering relation $m_1 < \ldots < m_k$ can be made arbitrarily large. This means that one of $A_1, \ldots, A_{2^k-1}$ must contain arbitrarily many pairs of elements with the same sum, which contradicts that it is a $B_2[g']$ set for some fixed $g'$.

\end{proof}

\section{Our Main Construction}
\hspace*{0.5cm} We now give our main construction, which improves upon our initial construction as described in the last section. Our previous construction had the undesirable feature that our value of $g$ grew as function of our union size. This was due to the fact that a sum of two elements both from the same $S_j$ uniquely determined the pairs of values from each of the sequences $X_1, \ldots, X_k$ going into it, but these could be recombined arbitrarily to get another occurrence of the same sum. We will overcome this problem by introducing an error correcting code, which will enforce that the occurrence of the sum is unique. We do not need to adapt our Ramsey theory argument to this more complex situation, since an alternative counting argument replaces it.

We construct, for each positive integer $k$, a union of $k$ $B_2[1]$ sets which is not a union of $k-1$ $B_2[g]$ sets for any finite $g$. This resolves the variant of Erd\H{o}s' problem mentioned above, showing that one cannot always reduce the union size of a finite union of $B_2[1]$ sets, even if one is willing to use $B_2[g]$ sets for an arbitrarily high $g$. Extending this result to $B_h[g]$ sets for values of $h > 2$ is an interesting problem which we do not address.

We begin by defining $k$ vectors $v_1, \ldots, v_k \in \{+1, -1\}^d$ with two key properties. First, we require that for each $i \neq j$, $v_i+v_j$ has $> \frac{d}{2}$ coordinates equal to 0 (in other words, these vectors form an error correcting code with relative distance strictly greater than $\frac{1}{2}$). Second, we require the values $v_i + v_j$ to be distinct (i.e. $v_i + v_j = v_h + v_\ell$ holds if and only if the sets $\{i,j\}$ and $\{h, \ell\}$ are equal). Such vectors can be easily constructed from Hadamard matrices when $d = 2^j - 1$ for some $j$ such that $2^j \geq k$.

\begin{lemma}
\label{lem:vectorSumDistinctness}
For any fixed positive integer $k$ and for $d = 2^j -1$ such that $2^j \geq k$, there exist vectors $v_1, \ldots, v_k \in \{1,-1\}^d$ such that the pairwise vector sums $v_i + v_j$ are distinct, and have $> \frac{d}{2}$ 0's when $i \neq j$.
\end{lemma}

\begin{proof}
We let $H$ be a $2^j \times 2^j$ Hadamard matrix with all 1's in its first column (these can be recursively constructed, and are also known as Walsh matrices). This matrix has entries in $\{1,-1\}$, and any two distinct rows are orthogonal. We take $v_1, \ldots, v_k$ to be the first $k$ rows of $H$, where we omit from each the first column's entry, which is always equal to 1. These are distinct vectors of length $d = 2^j-1$, and we claim that each $v_i + v_j$ for $i\neq j$ has $> \frac{d}{2}$ 0's. To see why, we note that $v_i \cdot v_j = -1$ (because the rows of $H$ are orthogonal and we have omitted the initial 1's), and each coordinate of $v_i, v_j$ contributes $1$ to $v_i \cdot v_j$ if $v_i$ and $v_j$ are equal in this coordinate, and contributes -1 if they are unequal. Hence, $v_i$ and $v_j$ must be unequal in strictly more than half the coordinates, so $v_i + v_j$ has $> \frac{d}{2}$ 0's.

We now suppose that $v_i + v_j = v_h + v_\ell$ and that $i\notin \{h, \ell\}$. Then we have:
\[v_i \cdot (v_h + v_\ell) = v_i \cdot v_h + v_i \cdot v_\ell = -1-1= -2.\]
However,
\[v_i \cdot (v_i + v_j)= v_i \cdot v_i + v_i \cdot v_j = d -1 > -2,\]
so we have a contradiction. Thus, $i \in \{h, \ell\}$. It follows that $\{i,j\} = \{h, \ell\}$.
\end{proof}

We now define $d$ disjoint sequences of positive integers, $X_1 = \{x^1_i\}_{i=1}^\infty, X_2 = \{x^2_i\}_{i=1}^\infty,$ $\ldots$, $X_d = \{x^d_i\}_{i=1}^\infty$, where each consists of powers of 5. For concreteness, we take $X_j$ to be the sequence $\{5^{id+j}\}_{i=1}^{\infty}$ for each $j$. We additionally define an infinite set $S \subset \mathbb{N}^d$ as follows. We let $M$ be the $d \times \lceil \frac{d}{2} \rceil$ Vandermonde matrix:
\[M = \left(
        \begin{array}{ccccc}
          1 & 1 & 1 & \ldots & 1 \\
          1 & 2 & 2^2 & \ldots & 2^{\lceil \frac{d}{2} \rceil -1} \\
          \vdots & \vdots & \vdots & \ddots & \vdots \\
          1 & d & d^2 & \ldots & d^{\lceil \frac{d}{2} \rceil -1} \\
        \end{array}
      \right)
\]
We note that any $\lceil \frac{d}{2} \rceil$ rows of the matrix form an invertible $\lceil \frac{d}{2} \rceil \times \lceil \frac{d}{2} \rceil$ Vandermonde matrix. We also note that invertibility remains even if we reduce the entries modulo any prime which is $> d$ (because $1, \ldots, d$ will have distinct modular reductions). By Bertrand's Postulate, we know such a prime exists which is $\leq 2d$. Hence, we obtain a reduced matrix $M$ with positive entries $ < 2d$ such that any $\lceil \frac{d}{2} \rceil$ rows form an invertible matrix (invertible over $\mathbb{R}$).

We now define $S$ as:
\[S := \big \{ M \cdot (i'_1, \ldots, i'_{\lceil \frac{d}{2} \rceil})^t : (i'_1, \ldots, i'_{\lceil \frac{d}{2} \rceil}) \in \mathbb{N}^{\lceil \frac{d}{2} \rceil} \big \}.\] (We use the notation $(i'_1, \ldots, i'_{\lceil \frac{d}{2} \rceil})^t$ to denote the transpose, i.e. $(i'_1, \ldots, i'_{\lceil \frac{d}{2} \rceil})^t$ denotes a column vector whose first entry is $i'_1$, etc.) The key property of $S$ that we will use is that if we are given at least half of the coordinates of some tuple $(i_1, \ldots, i_d) \in S$, we can uniquely solve for the remaining coordinates (by solving a linear system of $\lceil \frac{d}{2} \rceil$ linearly independent equations in $\lceil \frac{d}{2} \rceil$ unknowns). In other words, $S$ is an error-correcting code. (More precisely, a Vandermonde matrix modulo a prime $p$ is the generating matrix for a Reed-Solomon code over $\mathbb{F}_p$.)

For each $j$ from 1 to $k$, we define $W_j \subset \mathbb{Z}$ as:
\[W_j := \{ (x^1_{i_1}, \ldots, x^d_{i_d}) \cdot v_j : (i_1, \ldots, i_d) \in S\}.\]
In other words, an element of $W_j$ is formed by taking a $d$-tuple in $S$, using the coordinates as indices into the $d$ disjoint sequences $X_1, \ldots, X_d$, and taking the linear combination of the corresponding values with coefficients equal to the coordinates of $v_j$.

We will prove that each $W_j$ is a $B_2[1]$ set, and that $W := W_1 \cup W_2 \cup \ldots \cup W_k$ is a union of $k$ $B_2[1]$ sets that cannot be decomposed as a union of $k-1$ $B_2[g]$ sets for any finite value of $g$. (We note that $W$ and $S$ are defined with respect to a fixed $k$, and we leave this dependence implicit. In other words, $W$ and $S$ actually represent a family of constructions, parameterized by $k$.) We start by proving some useful lemmas.

\begin{lemma}
\label{lem:uniqueExpressions}
Each element of $W$ has a unique expression as $(x^1_{i_1}, \ldots, x^d_{i_d}) \cdot v_j$ for $(i_1, \ldots, i_d) \in S$ and $1 \leq j \leq k$. In particular, the sets $W_j$ are disjoint.
\end{lemma}

\begin{proof} This simply follows from the fact that base 5 expansions of integers with coefficients in $\{-2,-1,0,1,2\}$ are unique. Any value of the form $(x^1_{i_1}, \ldots, x^d_{i_d}) \cdot v_j$ has a base 5 expansion with coefficients in $\{-1,0,1\}$. From this expansion, we can uniquely determine the values of $x^1_{i_1}, \ldots, x^d_{i_d}$ and the coordinates of $v_j$.
\end{proof}

Next, we will obtain a precise characterization of the repeated sums and differences in $W$. We start with the following lemma:

\begin{lemma}
\label{lem:sumSetDisjointness}
The sets $W_i + W_j$ ($1\leq i ,j \leq k$) are disjoint. In other words, $W_i + W_j$ intersects $W_h + W_\ell$ if and only if $\{i,j\}$ and $\{h, \ell\}$ are equal.
\end{lemma}

\begin{proof} Again, this follows from the fact that base 5 expansions of integers with coefficients in $\{-2, -1, 0, 1,2\}$ are unique. We suppose that $\{i,j\} \neq \{h, \ell\}$, so (from Lemma \ref{lem:vectorSumDistinctness}) we have that $v_i + v_j \neq v_h + v_\ell$. Without loss of generality, we suppose that $v_i + v_j$ and $v_h + v_\ell$ differ in the first coordinate. We suppose that $W_i + W_j$ intersects $W_h + W_\ell$. This means that there exist tuples $(i_1, \ldots, i_d), (j_1, \ldots, j_d), (h_1, \ldots, h_d), (\ell_1, \ldots, \ell_d) \in S$ such that:
\[(x^1_{i_1}, \ldots, x^d_{i_d}) \cdot v_i + (x^1_{j_1}, \ldots, x^d_{j_d}) \cdot v_j = (x^1_{h_1}, \ldots, x^d_{h_d}) \cdot v_h + (x^1_{\ell_1}, \ldots, x^d_{\ell_d}) \cdot v_\ell.\]

Since base 5 expansions with coefficients in $[-2,2]$ are unique, we must have the same contribution of terms from sequence $X_1$ on both sides. This can only occur when the set of the first coordinates of $v_i, v_j$ and the set of the first coordinates of $v_h, v_\ell$ contain the same number of +1's and -1's, i.e. when $v_i + v_j$ and $v_h + v_\ell$ agree in the first coordinate. This contradicts our assumption that $v_i + v_j$ and $v_h + v_\ell$ differ in the first coordinate, so we have shown that $W_i + W_j$ and $W_h + W_\ell$ are disjoint when $v_i + v_j \neq v_h + v_\ell$, i.e. when $\{i,j\} \neq \{h, \ell\}$.
\end{proof}

We now prove a very helpful general lemma. We let $\phi:S\rightarrow \Z^d$ denote the map which takes a $d$-tuple $(i_1, \ldots, i_d)$ in $S$ to the vector $(x^1_{i_1}, \ldots, x^d_{i_d}) \in \Z^d$. We note that each element of our set $W$ can be expressed as $\phi(M\cdot y)\cdot v_i$ for some $i$ and some vector $y \in \Z^{\lceil \frac{d}{2} \rceil}$, where $M$ is the matrix described above.
\begin{lemma}
\label{lem:general}
We let $v'_i$ and $v'_j$ denote \textbf{any two vectors} in $\{+1, -1\}^d$. We suppose that $y, z, y', z' \in \Z^{\lceil \frac{d}{2} \rceil}$ satisfy:
\[\phi(M \cdot y) \cdot v'_i + \phi(M \cdot z) \cdot v'_j = \phi(M \cdot y') \cdot v'_i + \phi(M \cdot z') \cdot v'_j.\]
If $v'_i + v'_j$ is equal to 0 in $\geq \frac{d}{2}$ coordinates, then either $y = y'$ and $z = z'$ or $y = z$ and $y' = z'$. If $v'_i + v'_j$ is non-zero in $\geq \frac{d}{2}$ coordinates, then either $y = y'$ and $z = z'$ or $y = z'$ and $z = y'$.
\end{lemma}

\begin{proof} We let $C$ denote the value $\phi(M \cdot y) \cdot v'_i + \phi(M \cdot z) \cdot v'_j$, which is equal to $\phi(M \cdot y') \cdot v'_i +  \phi(M \cdot z') \cdot v'_j$. We consider the base 5 expansion of $C$ with coefficients in $[-2,2]$. We let $n \in [d]$ denote a coordinate where $v'_i + v'_j $ is equal to 0. If the base 5 expansion of $C$ includes no terms from the sequence $X_n$, we may conclude that the $n^{th}$ coordinate of $M\cdot y$ and the $n^{th}$ coordinate of $M \cdot z$ are equal. In other words, if we let $M_n$ denote the $n^{th}$ row of $M$, we have that $y - z$ is orthogonal to $M_n$, as is $y' - z'$. We let $EQUAL$ denote the set of coordinates $n$ where $v'_i + v'_j$ is equal to 0 and no terms from $X_n$ appear in our base 5 expansion of $C$. We let $Null(EQUAL)$ denote the space in $\mathbb{R}^{\lceil \frac{d}{2} \rceil}$ of vectors orthogonal to all the rows $M_n$ of $M$ for $n \in EQUAL$. Then we have shown so far that $y - z$ and $y' - z'$ are in $Null(EQUAL)$.

We now consider a coordinate $n\in [d]$ where $v'_i + v'_j=0$ but we see two terms (of opposite sign) from the sequence $X_n$ in the base 5 expansion of $C$. Since these terms have different signs, we can tell which came from dotting with $v'_i$ and which came from dotting with $v'_j$. Thus, we must have that the $n^{th}$ coordinate of $M \cdot y$ and the $n^{th}$ coordinate of $M \cdot y'$ are equal, and similarly, the $n^{th}$ coordinates of $M \cdot z$ and $M \cdot z'$ must be equal. Thus, $y - y'$ and $z-z'$ are both orthogonal to $M_n$. We define the set $SAME$ to include all such coordinates $n$, and we let $Null(SAME)$ denote the space in $\mathbb{R}^{\lceil \frac{d}{2} \rceil}$ of vectors orthogonal to all the rows $M_n$ of $M$ for $n \in SAME$. We have shown that $y-y', z-z' \in Null(SAME)$.

Next, we consider a coordinate $n \in [d]$ where $v'_i + v'_j \neq 0$. In such coordinates, we see two terms of the same sign from the sequence $X_n$ in the base 5 expansion of $C$. There are then two possibilities: either $y - y'$ and $z-z'$ are both orthogonal to $M_n$, or $y - z'$ and $z-y'$ are both orthogonal to $M_n$. If $y - y'$ and $z-z'$ are both orthogonal to $M_n$, we add $n$ to the set $SAME$. If this does not hold, then we must have $y-z'$ and $z-y'$ both orthogonal to $M_n$, and we define a new set $DIFF$ to include such coordinates $n$. We let $Null(DIFF)$ denote the space in $\mathbb{R}^{\lceil \frac{d}{2} \rceil}$ of vectors orthogonal to all the rows $M_n$ of $M$ for $n \in DIFF$. Then we have that $y-z', z-y' \in Null(DIFF)$. We note that we have defined the sets $EQUAL$, $SAME$, and $DIFF$ so that they are disjoint, and their union is $[d]$ (all of the $d$ coordinates).

We now examine 4 possible cases:
\begin{enumerate}
  \item $|EQUAL| \geq \frac{d}{2}$
  \item $|SAME| \geq \frac{d}{2}$
  \item $|DIFF| \geq \frac{d}{2}$
  \item  $|EQUAL|, |SAME|, |DIFF| \leq \frac{d}{2}$.
\end{enumerate}

In case 1., $y-z$ and $y'-z'$ are each orthogonal to at least $\frac{d}{2}$ rows of $M$, so we must have $y = z$ and $y'=z'$. In case 2., $y-y'$ and $z-z'$ are each orthogonal to at least $\frac{d}{2}$ rows of $M$, so we must have $y = y'$ and $z = z'$. In case 3., $y -z'$ and $z-y'$ are each orthogonal to at least $\frac{d}{2}$ rows of $M$, so we must have $y = z'$ and $z = y'$.

In case 4., we note that $y - y' + z - z' \in Null(SAME) \cup Null(DIFF)$, $y - z + y' - z' \in Null(EQUAL) \cup Null(DIFF)$, and $y -y'-z+z' \in Null(SAME) \cup Null(EQUAL)$. Since $|EQUAL|, |SAME|, |DIFF| \leq \frac{d}{2}$, we have that $|SAME\cup DIFF|, |EQUAL \cup DIFF|, |SAME \cup EQUAL|$ are all $\geq \frac{d}{2}$. Hence, we have that: $y-y' + z-z' = 0 = y-z+y'-z' = y - y'-z+z'$, which implies that $y = y' = z = z'$.

Now, if $v'_i + v'_j$ is equal to 0 in $\geq \frac{d}{2}$ coordinates, then being exclusively in case 3. is impossible. Thus, we may conclude that either $y = y'$ and $z = z'$ or $y = z$ and $y' = z'$. If $v'_i + v'_j$ is nonzero in $\geq \frac{d}{2}$ of the coordinates, then being exclusively in case 1. is impossible, so either $y = y'$ and $z = z'$ or $y = z'$ and $z = y'$.

\end{proof}

This lemma has a few useful corollaries:

\begin{corollary} Each $W_i$ is a $B_2[1]$-set.
\end{corollary}

\begin{proof} We apply Lemma~\ref{lem:general} with $v_i' = v_i$ and $v'_j = v_i$. Since $2v_i$ is nonzero in all $d$ coordinates, we can conclude that either $y =y'$ and $z= z'$ or $y = z'$ and $z = y'$. This means that if $a+b$ is a sum of two elements of $W_i$, the only other way to express it as a sum of two elements of $W_i$ is as $b+a$. Hence $W_i$ is a $B_2[1]$ set.
\end{proof}

\begin{corollary} $W$ is a $B^\circ_2[2]$-set.
\end{corollary}

\begin{proof} We suppose that we have $y, z, y', z'$ such that
\[\phi(M \cdot y) \cdot v_i - \phi(M\cdot z) \cdot v_j = \phi(M \cdot y') \cdot v_h - \phi(M \cdot z') \cdot v_\ell.\]
By the same argument employed in the proof of Lemma~\ref{lem:sumSetDisjointness}, this can only occur when $v_i - v_j = v_h - v_\ell$, i.e. when $v_i + v_\ell = v_h + v_j$. Since the sums of these vectors are unique, we must have either:
\begin{enumerate}
  \item $v_i = v_h$ and $v_j = v_\ell$ (and $i \neq j$) or
  \item $v_j = v_i$ and $v_h = v_\ell$.
\end{enumerate}
In case 1., we have: $\phi(M \cdot y) \cdot v_i - \phi(M\cdot z) \cdot v_j = \phi(M \cdot y') \cdot v_i - \phi(M \cdot z') \cdot v_j$. We then apply Lemma~\ref{lem:general} with $v'_i = v_i$ and $v'_j = -v_j$. Then $v'_i + v'_j$ is nonzero in more than half of the coordinates (since $i \neq j$), so either $y = y'$ and $z = z'$ or $y = z'$ and $z = y'$. This gives us at most two ways of representing this value as a difference of two elements of $W$.

In case 2., we have: $\phi(M \cdot y) \cdot v_i - \phi(M\cdot z) \cdot v_i = \phi(M \cdot y') \cdot v_h - \phi(M \cdot z') \cdot v_h$. We can rearrange this to be:
\[\phi(M \cdot y) \cdot v_i + \phi(M \cdot z')\cdot v_h = \phi(M \cdot z) \cdot v_i + \phi(M \cdot y') \cdot v_h.\]
We then apply Lemma~\ref{lem:general} with $v'_i = v_i$, and $v'_j = v_h$ and the roles of $y, z, y', z'$ appropriately exchanged. If $i = h$, then $v_i + v_h$ is nonzero in all of the coordinates. In this case, we conclude that either $y = z$ and $y' = z'$ (in which case, the difference $\phi(M \cdot y)\cdot v_i - \phi(M \cdot z)\cdot v_i$ is 0), or $y = y'$ and $z = z'$ (in which case, we are looking at the very same representation of the difference). Neither of these cases results in an alternate way of expressing a nonzero element as a difference of elements in $W$.

If $i\neq h$, then $v_i + v_h$ is 0 in more than half of the coordinates. We conclude that either $y = z$ and $y' = z'$ (again, the difference being represented is then equal to 0), or $y = z'$ and $z = y'$. In this case, we see that we may have two ways of representing a nonzero value as a difference of two elements of $W$. We then ask, could we have more? In other words, could we have distinct representations
\[ \phi(M \cdot y) \cdot v_i - \phi(M \cdot z) \cdot v_i = \phi(M \cdot z) \cdot v_h - \phi(M \cdot y) \cdot v_h = \phi(M \cdot u)\cdot v_\ell - \phi(M \cdot w) \cdot v_m\]
for some $u, w, v_\ell, v_m$ where $y \neq z$? We first note that $v_m = v_\ell$ must then hold, again by the agrument employed in Lemma~\ref{lem:sumSetDisjointness}.

This gives us:
\[ \phi(M \cdot y) \cdot v_i - \phi(M \cdot z) \cdot v_i = \phi(M \cdot z) \cdot v_h - \phi(M \cdot y) \cdot v_h = \phi(M \cdot u)\cdot v_\ell - \phi(M \cdot w) \cdot v_\ell.\]
Applying the argument above with $v_i$ and $v_\ell$ instead of $v_h$, we conclude that if $y \neq z$, we must have $u = z$ and $w = y$. However, if we apply the above argument to $v_h$ and $v_\ell$ instead, we conclude that $u =y $ and $w = z$. Since these must simultaneously hold, we get that $y = z$, which is a contradiction. Putting it all together, we have now proven that only 0 can be represented as a difference of two elements of $W$ in more than 2 ways, so $W$ is a $B^\circ_2[2]$ set.

\end{proof}

We now have a rather complete understanding of the sums and differences of $W$. We have shown that $W$ is a $B^\circ_2[2]$ set and is a union of $k$ $B_2[1]$ sets. We also know that $W$ is not a $B_2[g]$ set for any $g$, since Lemma~\ref{lem:general} does reveal some repeated sums in $W$. For each $i\neq j$, we can get many representations of a single integer as a sum of an element in $W_i$ and an element of $W_j$ by examining sums of the form $\phi(M \cdot y) \cdot v_i + \phi(M \cdot y) \cdot v_j$. The value of this sum will only depend on the coordinates of $M \cdot y$ for which $v_i + v_j$ is nonzero, and this is less than half of the coordinates. This means that the sum does not fully determine $y$: in fact, there are infinitely many values $y'$ such that $M \cdot y'$ will agree with $M \cdot y$ in these coordinates where $v_i + v_j \neq 0$. This shows that for $i \neq j$, $W_i \cup W_j$ is not a $B_2[g]$ set for any $g$. Lemma~\ref{lem:general} also tells us that these repeated sums of the form $\phi(M \cdot y) \cdot v_i + \phi(M \cdot y) \cdot v_j = \phi(M \cdot y')\cdot v_i + \phi(M \cdot y')\cdot v_j$ are the \emph{only repeated sums} in $W+W$. Essentially, this means that $W'+W'$ will still be large for any subset $W'$ of $W$, even though $W$ is not a $B_2[g]$ set for any $g$. In fact, $W$ is not a union of $k-1$ $B_2[g]$ sets for any $g$, which we prove next:

\begin{lemma}
\label{lem:decomposition}
$W := W_1 \cup W_2 \cup \ldots \cup W_k$ is not a union of $k-1$ $B_2[g]$ sets, for any finite $g$.
\end{lemma}

\begin{proof} We suppose that this is not true, i.e. there exist sets $A_1, \ldots, A_{k-1}$ such that $W = A_1 \cup A_2 \cup \ldots \cup A_{k-1}$, where each $A_i$ is a $B_2[g]$ set for some fixed $g$. We consider each $d$-tuple $(i_1, \ldots, i_d)$ in $S$. This corresponds to $k$ elements of $W$, namely $(x^1_{i_1}, \ldots, x^d_{i_d}) \cdot v_1, \ldots, (x^1_{i_1}, \ldots, x^d_{i_d}) \cdot v_k$. By the pigeonhole principle, some pair of these must belong to the same set $A_{\ell}$. This means we have a distinct way of achieving a sum of the form $(x^1_{i_1}, \ldots, x^d_{i_d}) \cdot v_i + (x^1_{i_1}, \ldots, x^d_{i_d}) \cdot v_j$ in $A_\ell + A_{\ell}$ (this is a distinct way of achieving this sum because elements of $W$ have unique representations as $(x^1_{i_1}, \ldots, x^d_{i_d}) \cdot v_i$ by Lemma~\ref{lem:uniqueExpressions}). We note that:
\[(x^1_{i_1}, \ldots, x^d_{i_d}) \cdot v_i + (x^1_{i_1}, \ldots, x^d_{i_d}) \cdot v_j = (x^1_{i_1}, \ldots, x^d_{i_d}) \cdot (v_i+v_j),\]
and that $(v_i + v_j)$ is 0 in $> \frac{d}{2}$ of the coordinates.

We consider tuples $(i_1, \ldots, i_d) \in S$ such that all of $i_1, \ldots, i_d$ are $\leq n$, for some fixed positive integer $n$. We first count how many of these tuples there are. We note that $(i_1, \ldots, i_d) = M \cdot (i'_1, \ldots, i'_{\lceil \frac{d}{2} \rceil})$ for some $(i'_1, \ldots, i'_{\lceil \frac{d}{2} \rceil}) \in \mathbb{N}^{\lceil \frac{d}{2} \rceil}$. Thus, each of $i_1, \ldots, i_d$ is a linear combination of the values $i'_1, \ldots, i'_{\lceil \frac{d}{2} \rceil}$, with positive coefficients all $\leq 2d$. Thus, if we choose any $i'_1, \ldots, i'_{\lceil \frac{d}{2} \rceil}$ values such that each is $\leq \frac{n}{2d \lceil \frac{d}{2} \rceil }$, we will have $i_1, \ldots, i_d \leq n$. This shows that there are at least $\left( \frac{n}{2d \lceil \frac{d}{2} \rceil }\right) ^{\lceil \frac{d}{2} \rceil}$ tuples $(i_1, \ldots, i_d) \in S$ such that all of $i_1, \ldots, i_d$ are $\leq n$.

As discussed above, each of these $d$-tuples in $S$ contributes a unique way of forming a sum $(x^1_{i_1}, \ldots, x^d_{i_d}) \cdot (v_i+v_j)$ in $A_{\ell} + A_\ell$ for some $A_\ell$. When all of $i_1, \ldots, i_d$ are $\leq n$, there are at most ${k \choose 2} n^{\lceil \frac{d}{2} \rceil -1}$ possibilities for the value of $(x^1_{i_1}, \ldots, x^d_{i_d}) \cdot (v_i+v_j)$. We can see this by noting that there are ${k \choose 2}$ possibilities for $v_i + v_j$, and each of them only has at most $\lceil \frac{d}{2} \rceil -1$ non-zero coordinates. In each such coordinate, we know our index value is at most $n$.

We note that $d$ is a fixed function of $k$, and we consider letting $n$ grow to infinity. Since $\left( \frac{n}{2d \lceil \frac{d}{2} \rceil }\right) ^{\lceil \frac{d}{2} \rceil}$ grows faster as a function of $n$ than ${k \choose 2} n^{\lceil \frac{d}{2} \rceil -1}$, and there are only $k$ possibilities for $A_\ell$, we must have that for any fixed $g$, there is some $A_\ell$ such that some element of $A_\ell + A_\ell$ can be expressed in $> g$ ways as a sum of two elements of $A_\ell$. This contradicts that $A_\ell$ is a $B_2[g]$ set. Hence we have proven that $W$ is not a union of $k-1$ $B_2[g]$ sets for any finite $g$.
\end{proof}

We have now shown:
\begin{theorem} $W \subseteq \Z$ is a union of $k$ $B_2[1]$ sets that cannot be decomposed as a union of $k-1$ $B_2[g]$ sets for any $g$. $W$ is also a $B^\circ_2[2]$ set.
\end{theorem}

By employing the same counting argument as above for a fixed $n$ (sufficiently large with respect to $k$ and $g$), we can restate our result in the context of finite sets. We let $W_{n,k}$ denote the finite subset of $W$ formed by restricting to tuples $(i_1, \ldots, i_d) \in S$ such that $i_1, \ldots, i_d \leq n$. (Here we make the dependence on $k$ explicit.)

\begin{theorem}
\label{thm:main}
For any positive integers $g$ and $k$, we can choose $n$ sufficiently large so that the finite set $W_{n,k} \subseteq \Z$ is a $B^\circ_2[2]$ set that is a union of $k$ $B_2[1]$ sets, but cannot be decomposed as a union of $k-1$ $B_2[g]$ sets.
\end{theorem}

\section{Adapting Our Construction for Mixed Unions}
In the previous section, we constructed a set $W\subset Z$ for each $k$ such that $W$ could not be decomposed as a union of $k-1$ $B_2[g]$ sets for any $g$. However, our $W$ is a $B^\circ_2[2]$ set, and we would like to arrive at a set in $\mathbb{Z}$ which cannot be decomposed as a mixed union of $k$ $B_2[g]$ and $B^\circ_2[g]$ sets for each $k$. Constructing such a set will put us well on our way toward obtaining an explicit counterexample to the weak anti-Freiman conjecture. To accomplish this, we will first adjust our techniques to obtain a $B_2[2]$ set $W^\circ \subseteq \Z$ for each $k$ that cannot be decomposed as a union of $k-1$ $B^\circ_2[g]$ sets for any $g$. We will then consider $W^\circ \times W$ in $\Z^2$ for each $k$, and show that this cannot be decomposed as a mixed union of $\frac{k}{3}-1$ $B_2[g]$ and $B^\circ_2[g]$ sets for any $g$.

For each positive integer $k$, we set $d = k$ and we let $v_j$ be the vector in $\{1,-1\}^d$ with a $-1$ in the $j^{th}$ coordinate and 1's in all other coordinates. We note that for $k \geq 5$, $v_j$ and $v_h$ will agree in $> \frac{d}{2}$ coordinates for all $1\leq j,h \leq k$. We define the sequences $X_1, \ldots, X_d$ and the set $S \subset \Z^d$ as in the previous section. For each $i$ from 1 to $k$, we define:
\[W^\circ_j := \{ (x^1_{i_1}, \ldots, x^d_{i_d}) \cdot v_j : (i_1, \ldots, i_d) \in S\}.\]
We define $W^\circ := W^\circ_1 \cup W^\circ_2 \ldots \cup W^\circ_k$. We now prove the relevant properties of $W^\circ$. The dependence of $W^\circ$ on $k$ is implicit.

\begin{lemma}
\label{lem:uniqueExpressions2}
Each element of $W^\circ$ has a unique expression as $(x^1_{i_1}, \ldots, x^d_{i_d}) \cdot v_j$ for $(i_1, \ldots, i_d) \in S$ and $1 \leq j \leq k$. In particular, the sets $W^\circ_j$ are disjoint.
\end{lemma}

\begin{proof} This is the same as the proof of Lemma~\ref{lem:uniqueExpressions}.
\end{proof}

\begin{lemma} For each $j$, $W^\circ_j$ is a $B^\circ_2[1]$ set.
\end{lemma}

\begin{proof} We can represent any element of $W^\circ_j$ as $\phi(M \cdot y) \cdot v_j$ for some vector $y \in Z^{\lceil \frac{d}{2} \rceil}$. We suppose that there are vectors $y, z, y', z'$ such that:
\[\phi(M \cdot y ) \cdot v_j - \phi(M \cdot z) \cdot v_j = \phi(M \cdot y')\cdot v_j - \phi(M \cdot z')\cdot v_j.\]
We now apply Lemma~\ref{lem:general} with $v'_i = v_j$ and $v'_j = -v_j$. Since $v_j - v_j$ is 0 in all of the coordinates, we conclude that either $y = y'$ and $z = z'$ (so we do not get a new way of representing the value as a difference) or $y = z$ and $y' = z'$ (in which case, we are representing 0). Therefore, every nonzero value can be represented in at most one way as a difference of two elements of $W^\circ_j$.
\end{proof}

\begin{lemma} For $k\geq 5$, $W^\circ$ is a $B_2[2]$ set.
\end{lemma}

\begin{proof} We note that the sums $v_i + v_j$ are distinct (e.g. $i$ and $j$ can be determined from the sum as the two coordinates where the sum is 0 for $i\neq j$). As shown in Lemma~\ref{lem:sumSetDisjointness}, this implies that the sets $W^\circ_i + W^\circ_j$ are disjoint. Therefore, it suffices to consider vectors $y, z, y', z' \in \Z^{\lceil \frac{d}{2} \rceil}$ such that:
\[\phi(M \cdot y) \cdot v_i + \phi(M \cdot z) \cdot v_j = \phi(M \cdot y') \cdot v_i + \phi(M \cdot z') \cdot v_j.\]
Now we can apply Lemma~\ref{lem:general} with $v'_i = v_i$ and $v'_j = v_j$. Since $k \geq 5$, $v_i + v_j$ will be nonzero in more than half the coordinates, so either $y = y'$ and $z = z'$ or $y = z'$ and $z = y'$. This gives us at most 2 ways of representing any value as a sum of two elements of $W^\circ$, so $W^\circ$ is a $B_2[2]$ set.
\end{proof}

\begin{lemma} For $k \geq 5$, $W^\circ$ cannot be decomposed as a union of $k-1$ $B^\circ_2[g]$ sets for any $g$.
\end{lemma}

\begin{proof} We suppose that this is not true, i.e. there exist sets $A^\circ_1, \ldots, A^\circ_{k-1}$ such that $W^\circ = A^\circ_1 \cup A^\circ_2 \cup \ldots \cup A^\circ_{k-1}$, where each $A^\circ_i$ is a $B^\circ_2[g]$ set for some fixed $g$. We consider each $d$-tuple $(i_1, \ldots, i_d)$ in $S$. This corresponds to $k$ elements of $W^\circ$, namely $(x^1_{i_1}, \ldots, x^d_{i_d}) \cdot v_1, \ldots, (x^1_{i_1}, \ldots, x^d_{i_d}) \cdot v_k$. By the pigeonhole principle, some pair of these must belong to the same set $A^\circ_{\ell}$. This means we have a distinct way of achieving a difference of the form $(x^1_{i_1}, \ldots, x^d_{i_d}) \cdot v_i - (x^1_{i_1}, \ldots, x^d_{i_d}) \cdot v_j$ in $A^\circ_\ell - A^\circ_{\ell}$ (this is a distinct way of achieving this difference because elements of $W^\circ$ have unique representations as $(x^1_{i_1}, \ldots, x^d_{i_d}) \cdot v_i$ by Lemma~\ref{lem:uniqueExpressions2}). We note that:
\[(x^1_{i_1}, \ldots, x^d_{i_d}) \cdot v_i - (x^1_{i_1}, \ldots, x^d_{i_d}) \cdot v_j = (x^1_{i_1}, \ldots, x^d_{i_d}) \cdot (v_i-v_j),\]
and that $v_i - v_j$ is 0 in all but 2 of the coordinates.

We consider tuples $(i_1, \ldots, i_d) \in S$ such that each of $i_1, \ldots, i_d \leq n$, for some fixed positive integer $n$. From the proof of Lemma~\ref{lem:decomposition}, we know there are at least $\left( \frac{n}{2d \lceil \frac{d}{2} \rceil }\right) ^{\lceil \frac{d}{2} \rceil}$ of these tuples.

As discussed above, each of these $d$-tuples in $S$ contributes a unique way of forming a difference $(x^1_{i_1}, \ldots, x^d_{i_d}) \cdot (v_i-v_j)$ in $A^\circ_{\ell} - A^\circ_\ell$ for some $A^\circ_\ell$. When all of $i_1, \ldots, i_d$ are $\leq n$, there are at most ${k \choose 2} n^{2}$ possibilities for the value of $(x^1_{i_1}, \ldots, x^d_{i_d}) \cdot (v_i-v_j)$. We can see this by noting that there are ${k \choose 2}$ possibilities for $v_i - v_j$, and each of them only has 2 nonzero coordinates. In each such coordinate, we know our index value is at most $n$.

We note that $d$ is a fixed function of $k$, and we consider letting $n$ grow to infinity. Since $\left( \frac{n}{2d \lceil \frac{d}{2} \rceil }\right) ^{\lceil \frac{d}{2} \rceil}$ grows faster as a function of $n$ than ${k \choose 2} n^{2}$, and there are only $k$ possibilities for $A^\circ_\ell$, we must have that for any fixed $g$, there is some $A^\circ_\ell$ such that some element of $A^\circ_\ell + A^\circ_\ell$ can be expressed in $> g$ ways as a difference of two elements of $A^\circ_\ell$. This contradicts that $A^\circ_\ell$ is a $B^\circ_2[g]$ set. Hence we have proven that $W^\circ$ is not a union of $k-1$ $B^\circ_2[g]$ sets for any finite $g$.
\end{proof}

By employing the same counting argument as above with a fixed $n$ (sufficiently large with respect to $k$ and $g$), we can state our result in the context of finite sets:
\begin{theorem}
\label{thm:main2}
For any positive integers $g$ and $k\geq 5$, there exists a finite $B_2[2]$ set $W^\circ_{n,k} \subseteq \Z$ such that $W^\circ_{n,k}$ is a union of $k$ $B^\circ_2[1]$ sets but cannot be decomposed as a union of $k-1$ $B^\circ_2[g]$ sets.
\end{theorem}

Here, $W^\circ_{n,k}$ is the finite subset of $W^\circ$ formed by restricting to tuples $(i_1, \ldots, i_d) \in S$ such that $i_1, \ldots, i_d \leq n$.

We now fix $k$ and $g$ and consider the set $W^\circ_{n,k} \times W_{n,k} \subseteq \Z^2$, where $n$ is chosen to be sufficiently large with respect to $k$ and $g$, and $W_{n,k}$ is defined as in Theorem~\ref{thm:main}.

\begin{theorem}
\label{lem:tensorDecomposition}
For each fixed $g$ and $k \geq 5$, there exists a sufficiently large $n$ such that $W^\circ_{n,k} \times W_{n,k} \subseteq \Z^2$ cannot be expressed as a union of $\leq \frac{k}{3}-1$ $B^\circ_2[g]$ and $B_2[g]$ sets.
\end{theorem}

\begin{proof} We let $k' := \frac{k}{3} -1$. We suppose that
\[W^\circ_{n,k} \times W_{n,k} = \left(\bigcup_{i=1}^{j} A_i\right) \bigcup \left(\bigcup_{i=j+1}^{k'} A_i^\circ\right),\] where each $A_i$ is a $B_2[g]$ set and each $A^\circ_i$ is a $B_2^\circ[g]$ set. We note that at least half of the elements of $W^\circ_{n,k} \times W_{n,k}$ must be contained in either the union of the $A_i$'s or the union of the $A^\circ_i$'s. We suppose that $\geq \frac{1}{2}$ the elements are contained in the $A_i$'s. This implies that there must exist some $a \in W^\circ_{n,k}$ such that at least half of the elements $a \times b$ for $b \in W_{n,k}$ are contained in the union of the $A_i$'s.

We let $S_n$ denote the set of $d$-tuples $(i_1, \ldots, i_d) \in S$ such that $i_1, \ldots, i_d \leq n$. We define $N:= |S_n|$, and we number these tuples from 1 to $N$. For each $j$ from 1 to $N$, we let $I_j$ denote the set of $k$ elements of $W_{n,k}$ corresponding to the tuple $j$. We suppose that for $(1-\alpha) N$ of these sets $I_j$, we have less than $\frac{k}{3}$ elements of $a \times I_j$ in the union of the $A_i$'s. Then $\alpha$ must satisfy:
\[(1-\alpha)N\left(\frac{k}{3}\right) + \alpha N k \geq \frac{1}{2} Nk
\Leftrightarrow (1-\alpha)\left(\frac{1}{3}\right) + \alpha \geq \frac{1}{2}
\Leftrightarrow \frac{1}{3} - \frac{\alpha}{3} + \alpha \geq \frac{1}{2}
\Leftrightarrow \alpha \geq \frac{1}{4}.\]

This means that for at least $\frac{1}{4}N$ values of $j$, we have at least $\frac{k}{3}$ elements of $a \times I_j$ in the union of the $A_i$'s. Now, there are at most $k' < \frac{k}{3}$ of the $A_i$'s, so for these tuples $j$, we must have that two distinct elements of $a \times I_j$ will be in the same $A_i$. Each of these will correspond to a distinct representation of one of ${k \choose 2} n^{\lceil \frac{d}{2} \rceil -1}$ possible sum values in $\Z^2$ (note that all of these will be equal to $2a$ in the first coordinate). Since this will occur at least
\[\frac{1}{4} N  \geq \frac{1}{4} \left( \frac{n}{2d \lceil \frac{d}{2} \rceil }\right) ^{\lceil \frac{d}{2} \rceil}\]
times, and there are only $k'$ $A_i$'s, we can choose $n$ large enough to contradict that each $A_i$ is a $B_2[g]$ set (note that $k, d, g$ are all fixed).

Similarly, if at least half of the elements of $W^\circ_{n,k} \times W_{n,k}$ are contained in the union of the $A^\circ_i$'s, then there must be some fixed $b \in W_{n,k}$ such that at least half of the elements of $W^\circ_{n,k} \times b$ are contained in the $A^\circ_i$'s. Then for at least $\frac{1}{4}$ of the $N$ $d$-tuples in $S_n$, we will get a distinct representation of one of ${k \choose 2} n^2$ values as a difference of two elements of some $A^\circ_i$. We can then choose $n$ large enough to contradict that each $A^\circ_i$ is a $B^\circ_2[g]$ set.
\end{proof}

\section{A Counterexample to the Weak Anti-Freiman Conjecture}

We now use our sets $W^\circ_{n,k} \times W_{n,k}$ to disprove the weak anti-Freiman conjecture (Conjecture~\ref{con:anti-Freiman}).
We first prove a lemma about $\Lambda(4)$ sets. This is essentially Lemma 4.30 from \cite{TV1}.

\begin{lemma}\label{LambdatoSum}
\label{lem:subsetConditionHelper}
Let $S \subset \Z^{d}$ such that $K_{4}(S)<\infty$. (Recall the definition of $K_4(S)$ from equation (\ref{LambdaDef}) in subsection \ref{sec:definitions}.) Furthermore if $(h_{1},h_{2})\in \{(2,0),(1,1)\}$, then for any finite $S' \subseteq S$,
\[|h_{1}S'-h_{2}S'| \geq \frac{|S'|^2}{\left(K_{4}(S)\right)^{4} }.\]
\end{lemma}

\begin{proof} First, from the definition of $K_{4}(S)$ we have
\begin{equation}\label{Size1}\left|\left|\sum_{\xi \in S'} e(\xi \cdot x) \right|\right|_{L^4}^{4} \leq \left(K_{4}(S) \right)^4 |S'|^2. \end{equation}
Now we also have that
\[\vectornorm{\sum_{\xi \in S'} e(\xi \cdot x)}_{L^4}^{4} = \vectornorm{\left(\sum_{\xi \in S'}e(\xi \cdot x) \right)^{h_{1}}  \left(\sum_{\xi \in S'}e(-\xi \cdot x) \right)^{h_{2}} }_{L^2}^{2}.\]
We let \[R_{2,0}(\nu) = |\left((\xi_{1},\xi_{2}) \in S' \times S') : \xi_{1}+\xi_{2}= \nu \right)|\] and we also let \[R_{1,1}(\nu) = |\left((\xi_{1},\xi_{2}) \in S' \times S') : \xi_{1}-\xi_{2}= \nu \right)|.\] We then have:
\[ \sum_{\nu \in \Z^{d}} R_{h_{1},h_{2}}^2(\nu) \leq \left(K_{4}(S)\right)^{4} |S'|^2.\]
We also note that:
\begin{equation}\label{size2}\sum_{\nu \in \Z^{d}} R_{h_{1},h_{2}}(\nu) = |S'|^{2}.
\end{equation}

Finally, by Cauchy-Schwarz, (\ref{Size1}), and the fact that $R_{h_{1},h_{2}}(\nu)$ is supported on the set $h_{1}S' - h_{2}S'$, we have
\[\sum_{\nu \in \Z^{d}} 1_{\{h_{1}S' - h_{2}S'\}}(\nu)  R_{h_{1},h_{2}}(\nu) \leq ||1_{\{h_{1}S' - h_{2}S'\}}||_{L^2(\Z^d)} \vectornorm{ R_{h_{1},h_{2}} }_{L^2(\Z^d)}\]
\begin{equation}\label{size3} \leq |h_{1}S' - h_{2}S'|^{1/2}\left(K_{4}(S)\right)^{2} |S'|.
\end{equation}
From (\ref{size2}) and (\ref{size3}), we have that
\[|S'|^2 \leq |h_{1}S'- h_{2}S'|^{1/2} \left(K_{4}(S)\right)^{2} |S'|, \]
which completes the proof.

\end{proof}

\begin{lemma} There is a universal constant $\delta > 0$ such that for any $k \geq 5$, for any finite subset $W'$ of $W^\circ \times W$ (recall that $W^\circ \times W$ is defined with respect to $k$), $|W'+W'| \geq \delta |W'|^2$ and $|W'-W'| \geq \delta |W'|^2$.
\end{lemma}

\begin{proof} We fix a value of $k \geq 5$. We note that $W^\circ$ is a $B_2[2]$ set, and hence it is a $\Lambda(4)$ set, with its $\Lambda(4)$ constant bounded independently of $k$. Similarly, $W$ is $B^\circ_2[2]$, so it is also a $\Lambda(4)$ set, with its $\Lambda(4)$ constant bounded independently of $k$. Thus, by Lemma~\ref{LambdaProd}, we conclude that $W^\circ \times W$ is also a $\Lambda(4)$ set, with its $\Lambda(4)$ constant bounded independently of $k$. By the lemma above, there exists $\delta > 0$ independent of $k$ such that $|W'+W'| \geq \delta |W'|^2$ and $|W'-W'| \geq \delta |W'|^2$ for any finite subset $W'$ of $W^\circ \times W$. (We note that this can be proved directly from the combinatorial properties of our construction, but we prefer this proof because it highlights the connection between the anti-Freiman problem and $\Lambda(4)$ sets.)
\end{proof}

\begin{theorem}\label{nolargeb2} We let $\delta$ be as above, so for every $n$ and $k\geq 5$, we have that $|W'+W'| \geq \delta |W'|^2$ and $|W'-W'| \geq \delta |W'|^2$ for all finite subsets $W'$ of $W^\circ_{n,k} \times W_{n,k}$. For every $g$ and $\delta'$, there exist $k$ and $n$ sufficiently large such that $W^\circ_{n,k} \times W_{n,k}$ does not contain either a $B_2[g]$ set or a $B^\circ_2[g]$ set of size $\geq \delta' |W^\circ_{n,k} \times W_{n,k}|$.
\end{theorem}

\begin{proof} We again let $N$ denote the size of $S_n$, so $|W_{n,k}| = |W^\circ_{n,k}| = k N$. We suppose we have $A \subseteq W^\circ_{n,k} \times W_{n,k}$ such that $|A| \geq \delta' |W^\circ_{n,k} \times W_{n,k}| = \delta' k^2 N^2$. We again number the tuples of $S_n$ as 1 to $N$. We let $I_j$ denote the set of $k$ elements of $W_{n,k}$ corresponding to tuple $j$ and we let $I^\circ_j$ denote the set of $k$ elements of $W^\circ_{n,k}$ corresponding to tuple $j$. We note that for some fixed $a \in W^\circ_{n,k}$, $A$ must contain at least $\delta' k N$ elements of $a \times W_{n,k}$. We consider the sets $a \times I_j$. We suppose that $1 - \gamma$ of them have $< \frac{\delta'}{2} k$ elements in $A$. Then $\gamma$ must satisfy:
\[ (1-\gamma)\frac{\delta'}{2} + \gamma \geq \delta' \Leftrightarrow \gamma \geq \frac{\frac{\delta'}{2}}{1 - \frac{\delta'}{2}}.\]

We can then set $\gamma = \frac{\frac{\delta'}{2}}{1 - \frac{\delta'}{2}}$, and we have that at least a $\gamma$-fraction of the $I_j$'s have at least $\frac{\delta'}{2} k$ elements of $a \times I_j$ in $A$. As long as we choose $k$ so that $\delta' \frac{k}{2} \geq 2$, these will lead to repeated sums in $A$. More precisely, each pair of distinct elements in $a \times I_j$ will sum to one of ${k \choose 2} n^{\lceil \frac{d}{2}\rceil -1}$ values, and there are at least ${\delta' \frac{k}{2} \choose 2} \gamma N$ such pairs in $A$. Since $N$ is a faster growing function of $n$ than $n^{\lceil \frac{d}{2}\rceil -1}$, we can choose $n$ sufficiently large to contradict that $A$ is a $B_2[g]$ set.

Similarly, there is some fixed $b \in W_{n,k}$ such that at least $\delta' k N$ elements of $W^\circ_{n,k} \times b$ are contained in $A$. We then have that at least a $\gamma$-fraction of the sets $I_j^\circ \times b$ have at least $\frac{\delta'}{2}k $ elements in $A$. This will lead to repeated differences in $A$: each pair of distinct elements in $I_j^\circ \times b$ will have a difference equal to one of ${k \choose 2} n^2$ values, and there are at least ${ \delta' \frac{k}{2} \choose 2} \gamma N$ such pairs in $A$. Since $N$ is a faster growing function of $n$ than $n^2$, we can choose $n$ sufficiently large to contradict that $A$ is a $B^\circ_2[g]$ set. Thus, if we choose $k$ so that $\delta' \frac{k}{2} \geq 2$ and $n$ sufficiently large with respect to $k$, $g$, $d$, $\delta'$, we have that $A$ cannot be a $B_2[g]$ set or a $B^\circ_2[g]$ set.
\end{proof}

This is a counterexample to Conjecture~\ref{con:anti-Freiman} in $\Z^2$. To obtain a counterexample in $\Z$, we can use $F_2$-isomorphisms, which are discussed in the next section. We note that each $W^\circ_{n,k} \times W_{n,k}$ is a finite set, and thus there is a $F_2$-isomorphic copy of this set inside $\Z$ by Lemma~\ref{F2iso} (which we prove in the next section). If this image in $\Z$ contained a large $B_2[g]$ or $B^\circ_2[g]$ set, then this would correspond to a $B_2[g]$ or $B^\circ_2[g]$ set in $W^\circ_{n,k} \times W_{n,k}$, which we know does not exist. (If two finite sets are $F_2$-isomorphic, then one is a $B_2[g]$ or $B^\circ_2[g]$ set if and only if the other one is as well, by Lemma~\ref{LambdaIso}, which is also proved in the next section.)

\section{$\Lambda(4)$ Sets}
\label{sec:reduction}
We provide an alternate counterexample to Rudin's question for $\Lambda(4)$ sets: we give an explicit set in $\Z$ that is a $\Lambda(4)$ set, but cannot be expressed as finite union of $B_2[g]$ sets for any $g$. However, one might also ask about $B^\circ_2[g]$ sets, since all $B^\circ_2[g]$ sets are $\Lambda(4)$ sets as well:

\begin{lemma}Let $S \subset \Z^{d}$ be a $B_{2}^{\circ}[g]$ set. Then for any function $f \in L^2(\T^d)$ such that $\hat{f}$ is supported on $S$, we have:
\[ \left|\left| \sum_{\xi \in S} \hat{f}(\xi)e(\xi \cdot x) \right|\right|_{L^4} \leq (1+g^2)^{1/4} ||f||_{L^2}.  \]
\end{lemma}
\begin{proof}
We note:
\[ \left|\left|\sum_{\xi \in S} \hat{f}(\xi)e(\xi \cdot x) \right|\right|_{L^4}^{4} =
 \left|\left|\sum_{\xi_{1} \in S} \hat{f}(\xi_{1})e(\xi_{1} \cdot x) \sum_{\xi_{2} \in S} \overline{\hat{f}(\xi_{2})}e(-\xi_{2} \cdot x) \right|\right|_{L^2}^{2}\]
\[=\sum_{\nu \in \Z^d} \left|\sum_{\xi_{1}-\xi_{2}=\nu} \hat{f}(\xi_{1})\overline{\hat{f}(\xi_{2})} \right|^2 = \left(\sum_{\xi \in S} |\hat{f}(\xi)|^2\right)^2 + \sum_{\nu \neq 0 } \left|\sum_{\xi_{1}-\xi_{2}=\nu} \hat{f}(\xi_{1})\overline{\hat{f}(\xi_{2})} \right|^2 \]
\[\leq \left(\sum_{\xi \in S} |\hat{f}(\xi)|^2\right)^2 + g^2 \sum_{\nu \neq 0} \max_{\substack{|\hat{f}(\xi_{1})\hat{f}(\xi_{2})|^2 \\ \xi_{1}-\xi_{2} = \nu}} |\hat{f}(\xi_{1})\hat{f}(\xi_{2})|^2 \]
\[\leq \left(\sum_{\xi \in S} |\hat{f}(\xi)|^2\right)^2 + g^2 \left( \sum_{\xi \in S}|\hat{f}(\xi)|^2 \right)^2 = \left(1+g^2\right) \left(\sum_{\xi \in S} |\hat{f}(\xi)|^2\right)^2. \]

\end{proof}

This shows that every $B^\circ_2[g]$ set is also a $\Lambda(4)$ set, so any finite union of $B_2[g]$ sets and $B^\circ_2[g]$ sets is also a $\Lambda(4)$ set. This raises a variant of Rudin's question: is every $\Lambda(4)$ set a finite union of $B_2[g]$ and $B^\circ_2[g]$ sets? The answer to this question is also no, and we give a $\Lambda(4)$ set in $\Z$ which cannot be decomposed as a finite mixed union of $B_2[g]$ and $B^\circ_2[g]$ sets. In this section, we describe how to obtain this from our combinatorial construction above and we prove the following stronger result:

\begin{theorem3}
There exists a $\Lambda(4)$ set $S$ such that for any fixed choice of $\delta>0$ and $g$, there exists a finite subset $A$ of $S$ such that no subset $A'$ of $A$ satisfying $|A'| \geq \delta |A|$ is a $B_{2}[g]$ or $B_{2}^{\circ}[g]$ set.
\end{theorem3}

We will need the following integral form of Minkowski's inequality (see \cite{HA2}, Theorem 202).

\begin{lemma}Let $f(x,y) \in L^{p}( \T^{d_{1}} \times \T^{d_{2}})$ be a complex-valued function. For $p>1$, we have that
\[ \left( \int_{\T^{d_{1}}}  \left| \int_{\T^{d_{2}}} f(x,y) dy \right|^{p} dx\right)^{1/p} \leq \int_{\T^{d_{2}}}  \left( \int_{\T^{d_{1}}} |f(x,y)|^p dx \right)^{1/p}dy.\]
\end{lemma}

\begin{lemma}\label{LambdaProd}Let $S_{1}$ and $S_{2}$ be $\Lambda(p)$ sets in $\Z^{d_{1}}$ and $\Z^{d_{2}}$ respectively ($p>2$). The direct product $S=S_{1}\times S_{2} \subseteq \Z^{d_1+d_2}$ is a $\Lambda(p)$ subset of $\Z^{d_{1}+d_{2}}$ with $\Lambda(p)$ constant equal to $K_{p}^{d_{1}+d_{2}}(S) = K_{p}^{d_{1}}(S_{1}) K_{p}^{d_{2}}(S_{2})$.
\end{lemma}

\begin{proof} We let $f(x,y) \in L^2(\T^{d_1+d_2})$, with $\hat{f}$ supported on $S_1 \times S_2 \subseteq \Z^{d_1+d_2}$.
First we notice that if we fix $x_{0} \in \T^{d_{1}}$, then the Fourier transform of the function $f(x_{0},y)$ is supported on $S_{2}$. Similarly, if we fix $y_{0} \in \T^{d_{2}}$, then $f(x,y_{0})$ is a function with Fourier transform supported on $S_{1}$.
We have:

\[\left(\int_{\T^{d_1+d_2}} |f(x,y)|^{p} dydx \right)^{1/p} =
\left(\int_{\T^{d_1}}\int_{\T^{d_2}} |f(x,y)|^{p} dydx \right)^{1/p} \]

\[ \leq K^{d_2}_{p}(S_{2}) \left(\int_{\T^{d_1}} \left(\int_{\T^{d_2}} |f(x,y)|^{2}dy \right)^{p/2} dx\right)^{1/p}
\leq K^{d_2}_{p}(S_{2})  \left(\int_{\T^{d_2}} \left(\int_{\T^{d_1}} |f(x,y)|^{p}dx \right)^{2/p} dy\right)^{1/2} \]

\[ \leq K_{p}^{d_{1}}(S_{1}) K_{p}^{d_{2}}(S_{2}) \left(\int_{\T^{d_2}} \int_{\T^{d_1}} |f(x,y)|^{2}dxdy\right)^{1/2}.\]

This establishes that $K_{p}^{d_{1}+d_{2}}(S) \leq K_{p}^{d_{1}}(S_{1})  K_{p}^{d_{2}}(S_{2})$. To see that $K_{p}^{d_{1}+d_{2}}(S) \geq K_{p}^{d_{1}}(S_{1})  K_{p}^{d_{2}}(S_{2})$, we can consider a sequence of functions $\{g_n\}$ with Fourier coefficients supported on $S_1$ with $\frac{||g_n||_{L^p}}{||g_n||_{L^2}}$ approaching $K_p^{d_1}(S_1)$ and a sequence of functions $\{h_n\}$ with Fourier coefficients supported on $S_2$ with $\frac{||h_n||_{L^p}}{||h_n||_{L^2}}$ approaching $K_p^{d_2}(S_2)$. If we then consider the functions $f_n(x,y) := g_n(x)h_n(y)$, we see that  $K_p^{d_1+d_2}(S) \geq K_{p}^{d_{1}}(S_{1})  K_{p}^{d_{2}}(S_{2})$.
\end{proof}

Let $G_{1}$ and $G_{2}$ be abelian groups, and $S$ a finite subset of $G_{1}$. We say a map $\tau: S \rightarrow G_{2}$ is a $F_{2}$-isomorphism if $\tau$ is injective and
\[\tau(a)+\tau(b)=\tau(c)+\tau(d) \Leftrightarrow a+b = c+d \]
\[\tau(a)-\tau(b)=\tau(c)-\tau(d) \Leftrightarrow a-b = c-d \]
for $a,b,c,d \in S$. We say that $S$ and $\tau(S)$ are $F_2$-isomorphic. We note that $\tau^{-1}$ is a $F_2$-isomorphism from $\tau(S)$ to $S$. However, $\tau$ is not an isomorphism in the full sense of group theory, since $S$ and $\tau(S)$ may not be groups. We will need the following lemmas concerning $F_2$-isomorphisms.

\begin{lemma}
If $S$ is a finite subset of $\Z^{d}$, then translation of $S$ by $\alpha \in \Z^{d}$ is a $F_{2}$-isomorphism.
\end{lemma}

\begin{proof} We define $\tau(a) := a+\alpha$ for all $a \in S$. Then, for any $a, b, c, d \in S$, we have:
\[\tau(a) + \tau(b) = \tau(c) + \tau(d) \Leftrightarrow a+ b + 2\alpha = c+ d + 2 \alpha \Leftrightarrow a+b = c+d,\]
\[\tau(a) - \tau(b) = \tau(c) - \tau(d) \Leftrightarrow a - b + \alpha - \alpha = c - d + \alpha - \alpha \Leftrightarrow a - b = c-d.\]
Hence, translation by a constant $\alpha$ is a $F_2$-isomorphism.
\end{proof}

\begin{lemma}\label{F2iso}Let $S \subset \Z^d$ be a finite set. Then there exists a $F_2$-isomorphism of $S$ into $\Z$.
\end{lemma}
\begin{proof} We let
\[M = 5 \max_{\vec{s} \in S}\left \{\max_{1\leq i \leq d}  |\vec{s}_{i}| \right \}.\]

We then define our $F_2$-isomorphism $\tau$ by:
\[\tau(\vec{s})=\sum_{i=1}^{d}\vec{s}_{i}M^i.\]
For any $\vec{s}, \vec{t} \in S$, we have:
\[\tau(\vec{s}) + \tau \left(\vec{t}\right) = \sum_{i=1}^d \vec{s}_i M^i + \sum_{i=1}^d \vec{t}_i M^i = \sum_{i=1}^d \left(\vec{s}_i + \vec{t}_i\right) M^i.\]

Now, the range of possible values taken by $\vec{s}_i + \vec{t}_i$ falls within $[-2\max_{1\leq i \leq d}  |\vec{s}_{i}|, 2\max_{1\leq i \leq d}  |\vec{s}_{i}|]$. By definition of $M$, this range is contained in $(-\frac{M}{2}, \frac{M}{2})$, so base $M$ expansions of integers with coefficients in this range are unique.

Hence, for other vectors $\vec{u}, \vec{v} \in S$, we will have $\tau(\vec{u}) + \tau(\vec{v}) = \tau(\vec{s}) + \tau \left(\vec{t}\right)$ if and only if $\vec{u} + \vec{v} = \vec{s} + \vec{t}$ in $\Z^d$. Similarly,
\[\tau(\vec{s}) - \tau \left(\vec{t}\right) = \sum_{i=1}^d \left(\vec{s}_i - \vec{t}_i\right) M^i,\]
and $\tau(\vec{u}) - \tau(\vec{v}) = \tau(\vec{s}) - \tau \left(\vec{t}\right)$ if and only if $\vec{u} - \vec{v} = \vec{s}-\vec{t}$.

\end{proof}

\begin{lemma}\label{LambdaIso} If $U \subset \Z^{d_{1}}$ and $V \subset \Z^{d_{2}}$ are $F_{2}$-isomorphic, then $K_{4}^{d_{1}}(U)=K_{4}^{d_{2}}(V)$.
\end{lemma}
\begin{proof}
We consider $f \in L^{2}(\T^{d_{1}})$ such that $\hat{f}$ is supported on $U$. As in equation (3) above, we have:
\[||f||_{L^4}^2 = \left( \sum_{\xi \in \Z^{d_1}} \left| \sum_{\substack{\nu_1 + \nu_2 = \xi \\ \nu_1, \nu_2 \in U}} \hat{f}(\nu_1) \hat{f}(\nu_2)\right|^2\right)^{\frac{1}{2}}.\]
We define $g \in L^2(\T^{d_{2}})$, a function such that $\hat{g}$ is supported on $V$, by $\hat{g}(\xi) = \hat{f}(\tau(\xi))$, where $\tau$ is an $F_2$-isomorphism from $V$ to $U$ (we let $\hat{g}(\xi)$ be 0 for $\xi \notin V$). Now we have:
\[||g||_{L^4}^2 = \left(\sum_{\xi \in \Z^{d_2}} \left| \sum_{\substack{\mu_1 + \mu _2 = \xi \\ \mu_1, \mu_2 \in V}} \hat{g} (\mu_1) \hat{g} (\mu_2)\right|^2\right)^{\frac{1}{2}} = \left(\sum_{\xi \in \Z^{d_2}} \left| \sum_{\substack{\mu_1 + \mu _2 = \xi \\ \mu_1, \mu_2 \in V}} \hat{f} (\tau(\mu_1)) \hat{f} (\tau(\mu_2))\right|^2\right)^{\frac{1}{2}}.\]
We can let $\nu_1$ denote $\tau(\mu_1)$ and $\nu_2$ denote $\tau(\mu_2)$, and since $\tau$ is a bijection between $V$ and $U$ that preserves sum relations, this can be rewritten as:
\[\left(\sum_{\xi \in \Z^{d_1}} \left| \sum_{\substack{\nu_1 + \nu _2 = \xi \\ \nu_1, \nu_2 \in U}} \hat{f} (\nu_1) \hat{f} (\nu_2)\right|^2\right)^{\frac{1}{2}} = ||f||_{L^4}^2.\]

Conversely, we could start with a function $f$ such that $\hat{f}$ is supported on $V$ and obtain $g$ with $\hat{g}$ supported on $U$ via $\hat{g}(\xi) = \hat{f}(\tau^{-1}(\xi))$. We would again obtain $||g||_{L^4}^2 = ||f||_{L^4}^2$. This shows that $K_{4}^{d_{1}}(U)=K_{4}^{d_{2}}(V)$.

\end{proof}

\begin{lemma}\label{LambdaIso} Let $U \subset \Z^{d_{1}}$ and $V \subset \Z^{d_{2}}$ be $F_{2}$-isomorphic ($U$ and $V$ are finite sets). For any fixed positive integer $g$, the following two statements are equivalent (a) $k$ is the smallest integer such that $U$ is the union of $k$ $B_{2}[g]$ sets, and (b) $k$ is the smallest integer such that $V$ is the union of $k$ $B_{2}[g]$ sets.  The analogous statement holds for $B_{2}^{\circ}[g]$ sets.
\end{lemma}
\begin{proof} We suppose that $\tau: U \rightarrow V$ is a $F_2$-isomorphism. We suppose that $U$ can be expressed as the union of $k$ $B_2[g]$ sets, say $A_1, \ldots, A_k$. We consider each $\tau(A_i)$ as a set in $V$. If this is not a $B_2[g]$ set, then we must have distinct pairs $\{a_1, b_1\}, \ldots, \{a_{g+1}, b_{g+1}\}$ in $\tau(A_i)$ such that:
\[ a_1 + b_1 = a_2 + b_2 = \ldots = a_{g+1} + b_{g+1}.\]
By the properties of $\tau$, we then have that
\[ \tau^{-1}(a_1) + \tau^{-1}(b_1) = \ldots = \tau^{-1}(a_{g+1}) + \tau^{-1}(b_{g+1})\]
holds in $A_i$, and the pairs $\{\tau^{-1}(a_1), \tau^{-1}(b_1)\}, \ldots, \{\tau^{-1}(a_{g+1}), \tau^{-1}(b_{g+1})\}$ are distinct in $A_i$, since $\tau$ is a bijection. This contradicts that $A_i$ is a $B_2[g]$ set. Hence, $\tau(A_i)$ must be a $B_2[g]$ set for each $i$, and $V$ is the union of the these sets. Thus, $V$ can also be expressed as the union of $k$ $B_2[g]$ sets. By reversing the roles of $U$ and $V$ and considering $\tau^{-1}$ in place of $\tau$, we also see that if $V$ is a union of $k$ $B_2[g]$ sets, then so is $U$. This proves the equivalence of the statements in the lemma. The same statement for unions of $B^\circ_2[g]$ holds by noting that $\tau$ also preserves difference relations.
\end{proof}

We will use the following inequality of Littlewood and Paley (see \cite{ST1}, for example):

\begin{lemma}\label{LittlewoodPaleyDecomposition}\emph{(Littlewood-Paley)}Let $f \in L^p(\mathbb{T})$ such that $f(x)= \sum_{\xi \in\mathbb{N}} \hat{f}(\xi) e(\xi x)$. Define $S_n := [2^{n},2^{n+1})$ for $n \in \mathbb{N}$. There exists, for $1<p<\infty$, a positive constant $c_{p}$ such that
\[c_p^{-1} \left|\left|\left( \sum_{n=1}^{\infty} \left| \sum_{\xi \in S_n} \hat{f}(\xi) e(\xi x) \right|^2 \right)^{1/2}\right|\right|_{L^{p}(\mathbb{T})} \leq \left|\left|\sum_{\xi \in \mathbb{Z}} \hat{f}(\xi) e(\xi x)\right|\right|_{L^{p}(\mathbb{T})} \]
\[\leq c_p \left|\left|\left( \sum_{n=1}^{\infty} \left| \sum_{\xi \in S_n} \hat{f}(\xi) e(\xi x) \right|^2 \right)^{1/2} \right|\right|_{L^{p}(\mathbb{T})}. \]
\end{lemma}

From Theorem~\ref{lem:tensorDecomposition} above, we obtain finite sets $W^\circ_{n,k} \times W_{n,k}$ in $\Z^2$ for each $k \geq 5 $ which cannot be decomposed as a mixed union of $\frac{k}{3}-1$ $B_2[k]$ and $B^\circ_2[k]$ sets in $\Z^2$, where each $W^\circ_{n,k}$ is a $B_2[2]$ set in $\Z$ and each $W_{n,k}$ is a $B^\circ_2[2]$ set in $\Z$. We drop the parameter $n$ from our notation in the lemma statement below, since $n$ is a function of $k$, i.e. any $n$ sufficiently large with respect to $k$ will do.

\begin{lemma}\label{L4Red}There exists a $\Lambda(4)$ subset of $\Z$ that cannot be decomposed as a finite (mixed) union of $B_{2}[g]$ and $B_{2}^{\circ}[g]$ sets.
\end{lemma}

\begin{proof}Let us write $C_{k}':= W^\circ_{k} \times W_{k} \subset \Z^{2}$. Now $W^{\circ}_{k}$ is a $B_{2}[2]$ set and $W_{k}$ is a $B^{\circ}_{2}[2]$ set. Thus, $W^\circ_{k}$ and $W_{k}$ are $\Lambda(4)$ sets with $\Lambda(4)$ constant bounded by some universal constant $D$, independent of $k$. It then follows from Lemma \ref{LambdaProd} that $C_{k}' \subset \Z^2$ is a $\Lambda(4)$ set with $\Lambda(4)$ constant at most $D^2$.

By Lemma \ref{F2iso}, we can find a finite subset of $\Z$ satisfying the same properties and having a $\Lambda(4)$ constant at most $D^2$. Let us denote this set as $C_{k}$.  Since the translation of $C_{k}$ by $\alpha \in \Z$ is a $F_{2}$-isomorphism, we may translate $C_{k}$ without affecting its $\Lambda(4)$ constant and without destroying the combinatorial properties established above.  We may thus assume that $C_{k} \subset [2^{\psi(k)}, 2^{\psi(k)+1})$ where $\psi(k):\N \rightarrow \N$ is injective and $C_{k}$ has $\Lambda(4)$ constant at most $D^2$.

We now appeal to the Littlewood-Paley inequality to show that $C= \cup_{k=5}^{\infty} C_{k} $ is a $\Lambda(4)$ set. Let $f(x) = \sum_{\xi \in C} \hat{f}(\xi)e(\xi x)$ such that $||f||_{L^{2}(\T)}<\infty$. Then

\[||f||_{L^{4}(\mathbb{T})} \leq c_4 \left|\left|\left( \sum_{n=5}^{\infty} \left| \sum_{\xi \in C_{n}} \hat{f}(\xi) e(\xi x) \right|^2 \right)^{1/2} \right|\right|_{L^{4}(\mathbb{T})}
\leq c_4 \left( \sum_{n=5}^{\infty} \left|\left| \sum_{\xi \in C_n} \hat{f}(\xi) e(\xi x)  \right|\right|_{L^{4}(\mathbb{T})}^2 \right)^{1/2}\] \[ \leq c_{4} \left( \sum_{n=5}^{\infty} \left( D^2 \left|\left|  \sum_{\xi \in C_n} \hat{f}(\xi) e(\xi x)  \right|\right|_{L^{2}(\mathbb{T})}\right)^2 \right)^{1/2} \leq c_4 D^2 ||f||_{L^2(\mathbb{T})}.\]

Lastly, we note that $C$ is not a finite union of $B_{2}[g]$ and $B_{2}^{\circ}[g]$ sets. To see this, notice that a partition of $C$ as a union of $j$ $B_{2}[j]$ sets and $j$ $B_{2}^{\circ}[j]$ sets would imply a partition of $C_{k}$ as a union of $j$ $B_{2}[j]$ sets and $j$ $B_{2}^{\circ}[j]$ sets, which, by construction is impossible for large enough $k$.
\end{proof}

Theorem 3 easily follows. The fact that for every $\delta >0$ and $g$ there exists a finite subset $A$ of our $\Lambda(4)$ set such that any subset $A' \subseteq A$ satisfying  $|A'| \geq \delta |A|$ is not a $B_{2}[g]$ or $B_{2}^{\circ}[g]$ set follows from the fact that this holds (by Theorem \ref{nolargeb2} above) for the sets $C_{k}':= W^\circ_{k} \times W_{k} \subset \Z^{2}$ when $k$ is sufficiently large, and that $C$ contains a $F_{2}$-isomorphic copy of these sets.

\section{Acknowledgements} The first draft of this paper was completed before we were aware of Meyer's solution to Rudin's question, and we thank Stefan Neuwirth for pointing us to this work. We also thank Jeffrey Vaaler and David Zuckerman for helpful discussions.

\texttt{A. Lewko, Department of Computer Science, The University of Texas at Austin}

\textit{alewko@cs.utexas.edu}
\vspace*{0.5cm}

\texttt{M. Lewko, Department of Mathematics, The University of Texas at Austin}

\textit{mlewko@math.utexas.edu}

\end{document}